\documentclass[11pt]{amsart}
\usepackage{amssymb}
\usepackage{mathtools}
\usepackage{color}

\newtheorem{theo}{Theorem} 
\newtheorem{mainconj}[theo]{Conjecture} 

\newtheorem{lemma}{Lemma}[section]
\newtheorem{prop}[lemma]{Proposition}

\newtheorem{theoint}[lemma]{Theorem}
\newtheorem{claim}[lemma]{Claim}

\theoremstyle{remark}
\newtheorem{remark}[lemma]{Remark}

\theoremstyle{definition}
\newtheorem{defi}[lemma]{Definition}

\newcommand{\lin}{\textsc{l}}

\newcommand{\NN}{\mathbb{N}}
\newcommand{\RR}{\mathbb{R}}

\newcommand{\eps}{\varepsilon}

\newcommand{\CCC}{\mathcal{C}}
\newcommand{\DDD}{\mathcal{D}}

\newcommand{\HHH}{\mathcal{H}}

\newcommand{\OOO}{\mathcal{O}}
\newcommand{\PPP}{\mathcal{P}}
\newcommand{\RRR}{\mathcal{R}}
\newcommand{\SSS}{\mathcal{S}}

\newcommand{\tmu}{\tilde{\mu}}

\newcommand{\indic}{1\!\!1}

\newcommand{\vU}{\vec{U}}

\newcommand{\hdot}{\dot{H}^1}

\newcommand{\EMPH}[1]{\medskip\noindent\textit{#1}.}

\DeclareMathOperator{\supp}{supp}

\newcommand{\profiles}{$\left(U^j_{\lin},\left\{t_{j,n},x_{j,n},\lambda_{j,n}\right\}_n\right)_{j\geq 1}$}

\newcommand{\ds}{\displaystyle}

\numberwithin{equation}{section} 

\title[Scattering profile for the energy-critical wave equation]{Scattering profile for global solutions of the energy-critical wave equation}
\author[T.~Duyckaerts]{Thomas Duyckaerts$^1$}
\author[C.~Kenig]{Carlos Kenig$^2$}
\author[F.~Merle]{Frank Merle$^3$}
\thanks{$^1$LAGA, Universit\'e Paris 13, Sorbonne Paris Cit\'e, UMR 7539. Partially supported by ERC Grant Blowdisol 291214 and French ANR Grant SchEq ANR-12-JS01-0005-01}
\thanks{$^2$University of Chicago. Partially supported by NSF Grants DMS-1265429 and DMS-1463746}
\thanks{$^3$Cergy-Pontoise (UMR 8088), IHES. Partially supported by ERC Grant Blowdisol 291214}

\date{\today}

\begin{document}

\begin{abstract}
 Consider the focusing energy-critical wave equation in space dimension $3$, $4$ or $5$. We prove that any global solution which is bounded in the energy space converges in the exterior of wave cones to a radiation term which is a solution of the linear wave equation. 
\end{abstract}

\maketitle

\section{Introduction}
In this note we consider the energy-critical nonlinear wave equation on $\RR^N$, $N\in \{3,4,5\}$:
 \begin{equation}
  \label{NLW}
  \left\{\begin{aligned}
  \partial_t^2u-\Delta u-|u|^{\frac{4}{N-2}}u&=0\\
(u,\partial_t u)_{\restriction t=0}&=(u_0,u_1)\in \hdot\times L^2,
\end{aligned}\right.
\end{equation}
where $\hdot=\hdot(\RR^N)$, $L^2=L^2(\RR^N)$, and $u$ is real-valued.
If $f$ is a function of space and time, we will denote $\vec{f}=(f,\partial_tf)$. It is known that the equation \eqref{NLW} is locally well posed in $\hdot\times L^2$ and that the energy
$$E(\vec{u}(t))=\frac 12 \int |\nabla u(t,x)|^2\,dx+\frac 12\int |\partial_tu(t,x)|^2\,dx-\frac{N-2}{2N}\int |u(t,x)|^{\frac{2N}{N-2}}\,dx$$
is conserved.

All solutions of the defocusing equation (\eqref{NLW} with a $+$ instead of a $-$ sign in front of the nonlinearity) are global and scatter to a solution of the linear equation (see e.g. \cite{Struwe88,Grillakis90b,GiSoVe92,Grillakis92,ShSt93,Kapitanski94,ShSt94,BaSh98,Nakanishi99b}). The dynamics of the focusing equation \eqref{NLW} is richer: small data solutions are global and scatter, however blow-up in finite time may occur \cite{Levine74}. Global, non-scattering solutions also exist. Examples of such solutions are given by solutions of the elliptic equation:
\begin{equation}
 \label{Ell}
 -\Delta Q=|Q|^{\frac{4}{N-2}}Q,\quad Q\in \dot{H}^{1},
\end{equation} 
 (see \cite{Ding86} for the existence of such solutions) and their Lorentz transforms
\begin{equation}
\label{travelling_waves}
Q_{\ell}(t,x)=Q\left(\left(-\frac{t}{\sqrt{1-|\ell|^2}}+\frac{1}{|\ell|^2} \left(\frac{1}{\sqrt{1-|\ell|^2}}-1\right)\ell\cdot x\right)\ell+x\right)=Q_{\ell}(0,x-t\ell)
\end{equation}
where $\ell \in \RR^N$, $|\ell|<1$. Note that 
$$Q_{\ell}(t,x)=Q_{\ell}(0,x-t\ell),$$
so  that $Q_{\ell}$ is a solitary wave traveling at speed $|\ell|$. The energy of $Q_{\ell}$ is given by:
\begin{equation}
\label{EQl}
E(\vec{Q}_{\ell}(0))=\frac{1}{\sqrt{1-|\ell|^2}}E(\vec{Q}(0))\underset{|\ell|\to 1}{\longrightarrow}+\infty. 
\end{equation} 
It is conjectured that any bounded, global solution of \eqref{NLW} is a sum of modulated, decoupled traveling waves and a scattering part. More precisely:
\begin{mainconj}[Soliton resolution]
\label{C:E1}
 Let $u$ be a solution of \eqref{NLW} on $[0,+\infty)\times \RR^N$ such that
 \begin{equation}
  \label{E2}
  \sup_{t\geq 0}\left\|\vec{u}(t)\right\|_{\hdot\times L^2}<\infty.
 \end{equation} 
 Then there exist a solution $v_{\lin}$ of the linear wave equation
\begin{equation}
 \label{E3}
 \left\{\begin{aligned}
( \partial_t^2-\Delta)v_{\lin}&=0\\
\vec{v}_{\lin\restriction_{t=0}}&=(v_0,v_1)\in \hdot\times L^2
\end{aligned}\right.
\end{equation} 
an integer $J\geq 0$, and for $j\in \{1,\ldots,J\}$, a (nonzero) traveling wave $Q_{\ell_j}^j$ ($|\ell_j|<1$), and parameters $x_j(t)\in \RR^N$, $\lambda_j(t)\in \RR^N$  such that
\begin{equation}
\label{expansion}
\lim_{t\to+\infty}
 \vec{u}(t)-\vec{v}_{\lin}(t)-\sum_{j=1}^J \left(\frac{1}{\lambda_j(t)^{\frac{N-2}{2}}}Q_{\ell_j}^j\left(0,\frac{\cdot-x_j(t)}{\lambda_j(t)}\right),\frac{1}{\lambda_j(t)^{\frac{N}{2}}}\partial_t Q_{\ell_j}^j\left(0,\frac{\cdot-x_j(t)}{\lambda_j(t)}\right)\right)=0
\end{equation} 
in $\hdot\times L^2$ and 
\begin{gather*}
 \forall j\in \{1,\ldots,J\},\quad \lim_{t\to\infty}\frac{x_j(t)}{t}=\ell_j,\quad \lim_{t\to\infty}\frac{\lambda_j(t)}{t}=0\\
 \forall j,k\in \{1,\ldots,J\},\quad j\neq k\Longrightarrow \lim_{t\to+\infty}\frac{|x_j(t)-x_k(t)|}{\lambda_j(t)}+\frac{\lambda_j(t)}{\lambda_k(t)}+\frac{\lambda_k(t)}{\lambda_j(t)}=+\infty.
\end{gather*}
\end{mainconj}
This conjecture was proved by the authors in \cite{DuKeMe13} for radial solutions in space dimension $N=3$ (in this case $x_j(t)\equiv 0$ and the only stationary solutions in the expansion \eqref{expansion} are $W$ and $-W$, where
$$W(x)=\frac{1}{\left(1+\frac{|x|^2}{3}\right)^{1/2}}.$$
In this article we extract the linear profile $v_{\lin}(t)$ which appears in the expansion \eqref{expansion}. More precisely, we prove:
\begin{theo}
\label{T:E1}
 Let $u$ be a solution of \eqref{NLW} on $[0,+\infty)\times \RR^N$ that satisfies \eqref{E2}.
Then there exists a solution $v_{\lin}$ of the linear wave equation \eqref{E3} such that, for all $A\in \RR$,
\begin{equation}
 \label{E4}
 \lim_{t\to+\infty} \int_{|x|\geq t+A}|\nabla(u-v_{\lin})(t,x)|^2+|\partial_t(u-v_{\lin})(t,x)|^2+\frac{1}{|x|^2}|u(t,x)|^2+|u(t,x)|^{\frac{2N}{N-2}}\,dx=0.
\end{equation} 
\end{theo}
We see Theorem \ref{T:E1} as a first step toward the proof of Conjecture \ref{C:E1}. It will be used in an article by Hao Jia and the authors to prove a weak form of the conjecture, i.e. that the expansion \eqref{expansion} holds for a sequence of times $t_n\to+\infty$ (see \cite{DuJiKeMe16P}). 
Note that Theorem \ref{T:E1} implies that the following limit exists:
$$\lim_{A\to-\infty} \lim_{t\to+\infty} \int_{|x|\geq t+A} \frac{1}{2}|\nabla u(t,x)|^2+\frac{1}{2}(\partial_tu(t,x))^2-\frac{N-2}{2N}|u(t,x)|^{\frac{2N}{N-2}}\,dx=E_{\lin}(v_0,v_1),$$
where 
$$E_{\lin}(\vec{v}_{\lin}(t))=\frac{1}{2}\int |\nabla v_{\lin}(t,x)|^2+(\partial_tv_{\lin}(t,x))^2\,dx$$
is the conserved energy for the linear wave equation. One can also prove (using for example the profile decomposition of \cite{BaGe99} recalled in \S \ref{SS:profiles}), as a consequence of Theorem \ref{T:E1}:
\begin{equation}
\label{weaklim}
\vec{S}_{\lin}(-t) \vec{u}(t) \xrightharpoonup[t+\infty]{}(v_0,v_1),
\end{equation} 
where $\vec{S}_{\lin}$ denotes the linear evolution (see \S \ref{SS:lin} below). 

We next mention a few related works. Theorem \ref{T:E1} is proved in \cite{DuKeMe12b} in the radial case in dimension $3$ (where it is significantly simpler). A very close proof yields Theorem \ref{T:E1} for radial solutions of \eqref{NLW} in higher dimensions (see e.g. \cite{CoKeLaSc15P}) and of the defocusing analogue of \eqref{NLW} with an additional linear potential \cite{JiaLiuXu14P}, still in the radial setting. It is proved in \cite{CoKeLaSc15b} by an adaptation of the proof of \cite{DuKeMe12b}. Related results for energy-critical, mass-supercritical Schr\"odinger equations were proved by T.~Tao in \cite{Tao04DPDE,Tao07DPDE,Tao08DPDE}.

\medskip

Let us give a short outline of the paper. We first prove (Section \ref{S:large_wave_cone}), as a consequence of small data theory and finite speed of propagation, that \eqref{E4} holds for large positive $A$. We then argue by contradiction, assuming that \eqref{E4} does not hold for all $A\in \RR$, and defining $\overline{A}$ as the largest real number such that \eqref{E4} does not hold. We divide (see Section \ref{S:sing_reg}) the elements of the sphere $S^{N-1}$ between \emph{regular directions} (in an angular neighbourhood of which \eqref{E4} holds locally, for some $A<\overline{A}$), and \emph{singular directions} (other elements of $S^{N-1}$), and prove, using geometrical considerations and again small data theory and finite speed of propagation, that the set of singular directions is finite. To conclude the proof, we show in Sections \ref{S:cond_profiles}, \ref{S:virial} and \ref{S:end_of_proof} that the set of singular directions is empty, which will contradict the definition of $\overline{A}$. The 
core of the proof of this fact is in Section \ref{S:virial}, where we prove, using virial type identities, that there are no nonlinear profiles remaining close to the wave cone $\{t=|x|\}$. This is coherent with the intuition that nonlinear objects with finite energy travel at a speed strictly slower than $1$, as the traveling waves \eqref{travelling_waves} and their energies \eqref{EQl} suggest. 
Let us also mention that we never need to know, in all this proof, what happens inside the wave cone (that is for $(t,x)$ such that $|t|-|x|\gg 1$). This is of course made possible by finite speed of propagation.

In Section \ref{S:preliminaries}, we give some preliminary results on linear and nonlinear wave equations. We introduce in particular an isometry between the initial data and the asymptotic profile of a solution of the linear wave equation that is known (see e.g. the work of Friedlander \cite{Friedlander80}) but seems to have been somehow forgotten. We construct this isometry (that we use many times in the article) in Appendices \ref{S:radiation} and \ref{S:funct} for the sake of completeness.
\section{Preliminaries and notations}
\label{S:preliminaries}
\subsection{Linear wave equation}
\label{SS:lin}
If $(v_0,v_1)\in \hdot\times L^2$, we let 
$$ S_{\lin}(t)(v_0,v_1):=\cos(t\sqrt{-\Delta})v_0+\frac{\sin(t\sqrt{-\Delta})}{\sqrt{-\Delta}}v_1$$
the solution $v_{\lin}$ of \eqref{E3}, and 
$$\vec{S}_{\lin}(t)(v_0,v_1):=\vec{v}_{\lin}(t)=\left(S_{\lin}(t)(v_0,v_1),\frac{\partial}{\partial t}\left(S_{\lin}(t)(v_0,v_1)\right)\right).$$
The linear energy 
$$ E_{\lin}(\vec{v}_{\lin}(t))=\frac 12\int |\nabla v_{\lin}(t,x)|^2\,dx+\frac 12\int (\partial_t v_{\lin}(t,x))^2\,dx$$
is conserved.
We will often use radial coordinates, denoting, for $x\in \RR^N\setminus \{0\}$, $r=|x|>0$ and $\omega=x/|x|\in S^{N-1}$. We also denote
$$\partial_r=\frac{\partial}{\partial_r}=\frac{x}{|x|}\cdot\nabla,\quad \frac{1}{r}\nabla_{\omega}=\nabla- \frac{x}{|x|}\partial_r.$$
We will often use the following asymptotic property.
\begin{theoint}
\label{T:B2}
 Assume $N\geq 3$ and let $v_{\lin}$ be a solution of the linear wave equation \eqref{E3}. Then
 \begin{equation}
  \label{B4}
  \lim_{t\to+\infty} \left\|\frac{1}{r}\nabla_{\omega}v_{\lin}(t)\right\|_{L^2}+\left\|\frac 1r v_{\lin}(t)\right\|_{L^2}=0
 \end{equation} 
 and there exists a unique $G_+\in L^2(\RR\times S^{N-1})$ such that
 \begin{align}
  \label{B1}
  \lim_{t\to+\infty} \int_0^{+\infty} \int_{S^{N-1}} \left|r^{\frac{N-1}{2}} \partial_t v_{\lin}(t,r\omega)-G_+(r-t,\omega)\right|^2\,d\omega\,dr&=0\\
  \label{B1bis}
  \lim_{t\to+\infty} \int_0^{+\infty} \int_{S^{N-1}} \left|r^{\frac{N-1}{2}} \partial_r v_{\lin}(t,r\omega)+G_+(r-t,\omega)\right|^2\,d\omega\,dr&=0.
 \end{align} 
 Furthermore,
 \begin{equation} 
  \label{B2}
  E_{\lin}(v_0,v_1)=\int_{\RR\times S^{N-1}} |G_+(\eta,\omega)|^2d\eta d\omega=\left\|G_+\right\|^2_{L^2}
 \end{equation}
and the map
\begin{align*}
 (v_0,v_1)&\mapsto \sqrt{2}G_+\\
(\dot{H}^1\times L^2)(\RR^N)&\to L^2(\RR\times S^{N-1})
 \end{align*}
is a bijective isometry.
\end{theoint}

\begin{remark}
Using Theorem \ref{T:B2} on the solution $(t,x)\mapsto v_{\lin}(-t,x)$ of \eqref{E3}, we obtain that 
\begin{equation*}
  \lim_{t\to-\infty} \left\|\frac{1}{r}\nabla_{\omega}v_{\lin}(t)\right\|_{L^2}+\left\|\frac 1r v_{\lin}(t)\right\|_{L^2}=0,
 \end{equation*} 
 that there exists $G_-\in L^2(\RR\times S^{N-1})$ such that
\begin{align*}
  \lim_{t\to -\infty} \int_0^{+\infty} \int_{S^{N-1}} \left|r^{\frac{N-1}{2}} \partial_t v_{\lin}(t,r\omega)-G_-(t+r,\omega)\right|^2\,d\omega\,dr&=0\\
  \label{B1bis}
  \lim_{t\to-\infty} \int_0^{+\infty} \int_{S^{N-1}} \left|r^{\frac{N-1}{2}} \partial_r v_{\lin}(t,r\omega)-G_-(t+r,\omega)\right|^2\,d\omega\,dr&=0.
 \end{align*} 
 and that the map $(v_0,v_1)\mapsto \sqrt{2} G_-$ is a bijective isometry from $(\dot{H}^1\times L^2)(\RR^N)$ to $L^2(\RR\times S^{N-1})$
\end{remark}
We will call $G_+$ (respectively $G_-$) the \emph{radiation fields} associated to  $v_{\lin}$.  

Theorem \ref{T:B2} is known (see in particular the works of Friedlander \cite{Friedlander62,Friedlander80}). We give a proof in Appendices \ref{S:radiation} and \ref{S:funct} for the sake of completeness. Let us mention that the following identity for the conserved momentum is also available (but will not be used in this article):
$$\int \nabla v_0 v_1=-\int_{S^{N-1}\times \RR} \omega |G_+(\eta,\omega)|^2\,d\eta.$$

\subsection{Strichartz estimates}
 If $\Omega$ is a measurable subset of $\RR_t\times \RR^N_x$, of the form $\Omega=\bigcup_{t\in \RR}\{t\}\times \Omega_t$, $\Omega_t\subset \RR^N$ measurable, and $u$ a measurable function defined on $\Omega$, we denote
 \begin{equation}
  \label{E8}
  \|u\|_{S(\Omega)}:=\left( \int_{-\infty}^{+\infty}\left( \int_{\Omega_t} |u|^{\frac{2(N+2)}{N-2}}\,dx \right)^{\frac 12}\,dt \right)^{\frac{N-2}{N+2}}.
 \end{equation} 
 If $I\subset \RR$ is measurable, we will abuse notation, writing $S(I)$ for $S(I\times \RR^N)$.

The spaces $S(I)$ appear in the following Strichartz estimate (see \cite{GiVe95}): let $f \in L^1\left(\RR,L^2(\RR^N)\right)$, $(u_0,u_1)\in (\hdot\times L^2)(\RR^N)$, and 
   $$u(t)=S_{\lin}(t)(u_0,u_1)+ \int_{0}^t S_L(t-s)(0,f(s))\,ds$$
   the solution of
 \begin{equation}
   \label{E10}
( \partial_t^2-\Delta)u=f,\qquad 
\vec{u}_{\restriction_{t=0}}=(u_0,u_1)\in \hdot\times L^2.
\end{equation}
Then $u$ is well-defined, $\vec{u}\in C^0(\RR,\hdot \times L^2)$, $u\in S(\RR)$ and
$$\|u\|_{S(\RR)}+\sup_{t\in \RR} \|\vec{u}(t)\|_{\hdot\times L^2}\leq C\left(\|f\|_{L^1(\RR,L^2)}+\|(u_0,u_1)\|_{\hdot\times L^2}\right).$$
 We will use occasionally other Strichartz estimates: in the preceding inequality, one can replace $S(\RR)$ by $L^{\frac{2(N+1)}{N-2}}(\RR^{N+1})$, and also by $L^4(\RR,L^{12})$ (if $N=3$) or $L^2(\RR,L^{\frac{2N}{N-3}})$ (if $N\geq 4$).
 
 \subsection{Miscellaneous properties of the critical nonlinear wave equation}
 We recall the Cauchy theory for equation \eqref{NLW}:
 \begin{theo}
  \begin{enumerate}
   \item Small data theory: there exists $\delta_0>0$ such that if $I$ is an interval containing $0$ and $(u_0,u_1)\in \hdot\times L^2$ is such that $\|S_L(t)(u_0,u_1)\|_{S(I)}<\delta_0$, then there exists a unique solution of \eqref{NLW} $\vec{u}\in C^0(I,\hdot\times L^2)$. Furthermore
   $$\sup_{t\in I}\left\|\vec{u}(t)-\vec{S}_{\lin}(t)(u_0,u_1)\right\|_{\hdot\times L^2}+\left\|u-S_{\lin}(\cdot)(u_0,u_1)\right\|_{S(I)}\leq C\|S_{\lin}(\cdot)(u_0,u_1)\|^{\frac{N+2}{N-2}}_{S(I)}.$$
   \item If $(u_0,u_1)\in \hdot\times L^2$, there exists a unique maximal solution $u$ of \eqref{NLW}. Letting $(T_-,T_+)$ the maximal interval of existence of $u$, we have the following blow-up criterion:
   $$ T_+<\infty\Longrightarrow \|u\|_{S((0,T_+))}=+\infty.$$
   \item \label{scatter}If $u$ is a solution of \eqref{NLW} such that $u\in S\big((0,T_+)\big)$, then $T_+=+\infty$ and $u$ scatters for positive times: there exists a solution $u_{\lin}$ of \eqref{E3} such that 
   \begin{equation}
    \label{E9} \lim_{t\to+\infty}\left\|\vec{u}(t)-\vec{u}_{\lin}(t)\right\|_{\hdot\times L^2}=0.
   \end{equation} 
  \end{enumerate}
 \end{theo}
 (see \cite{KeMe08}). In the theorem, a \emph{solution} of \eqref{NLW} on $I$ is by definition a solution in the Duhamel sense which is in $S(J)$ for all $J\Subset I$.
 Point \eqref{scatter} can be seen as the consequence of the following result on the non-homogeneous linear wave equation, that we will use repeatedly in the paper:
 \begin{claim}
  \label{Cl:E4}
Let $f\in L^1(\RR,L^2(\RR^N))$ and $u$ the solution of \eqref{E10} (in the Duhamel sense). 
 Then, there exists a solution $v_{\lin}$ of \eqref{E3} such that
   \begin{equation}
    \label{E11}
\lim_{t\to\infty} \|\vec{u}(t)-\vec{v}_{\lin}(t)\|_{\hdot\times L^2}=0.
    \end{equation} 
\end{claim}
\begin{proof}
The existence of $v_{\lin}$ follows from the formula
$$\vec{S}_{\lin}(-t)\vec{u}(t)=(u_0,u_1)+\int_0^t \vec{S}_{\lin}(-s)(0,f(s))\,ds$$
and energy estimates.
\end{proof}
We recall the \emph{finite speed of propagation} property: if $R>0$, $x_0\in \RR^N$ and $u$ is a solution of \eqref{E10} such that $(u_0,u_1)(x)=0$ for $|x-x_0|\leq R$, and $f(t,x)=0$ for $|x-x_0|\leq R-t$, $t\in [0,R]$, then $u(t,x)=0$ for $|x-x_0|\leq R-t$, $t\in [0,R]$. As a consequence, if the initial data of two solutions of \eqref{NLW} coincide for $|x-x_0|\leq R$, then the two solutions coincide for $t\in [0,R]$, $|x-x_0|<R-t$, if $t$ is in the domains of existence of both solutions. A consequence of finite speed of propagation and small data theory is the following claim:
\begin{claim}
 \label{Cl:FSP}
 There exists $\delta_1>0$ with the following property.
 Let $u$ be a solution of \eqref{NLW} such that $T_+(u)=+\infty$. Let $T>T_-(u)$, and $A>-T$. 
 \begin{enumerate}
  \item \label{I:NL_L}Assume $\|u\|_{S(\left\{|x|\geq A+t,\; t\geq T\right\})}=\delta<\delta_1.$
  Then 
  $\left\|S_{\lin}(\cdot-T)\vec{u}(T)\right\|_{S(\left\{|x|\geq A+t,\; t\geq T\right\})}\leq 2\delta.$
  \item \label{I:L_NL}Assume 
  $\left\|S_{\lin}(\cdot-T)\vec{u}(T)\right\|_{S(\left\{|x|\geq A+t,\; t\geq T\right\})}=\delta'<\delta_1.$
  Then
  $\|u\|_{S(\left\{|x|\geq A+t,\; t\geq T\right\})}\leq 2\delta'.$
 \end{enumerate}
\end{claim}
\begin{proof}
 Let $\tilde{u}$ be the solution of 
 \begin{equation*}
  \left\{\begin{aligned}
  \partial_t^2\tilde{u}-\Delta \tilde{u}&=|u|^{\frac{4}{N-2}}u\indic_{\{|x|\geq A+t\}}\\
(\tilde{u},\partial_t \tilde{u})_{\restriction t=T}&=\vec{u}(T).
\end{aligned}\right.
\end{equation*}
By Strichartz estimates, for all $T_1>T$,
\begin{equation*}
\left|\left\|\tilde{u}\right\|_{S(\left\{|x|\geq A+t,\; T\leq t\leq T_1\right\})}-\left\|S_{\lin}(\cdot-T)\vec{u}(T)\right\|_{S(\left\{|x|\geq A+t,\; T\leq t\leq T_1\right\})}\right|
\leq C\left\|u\right\|_{S(\left\{|x|\geq A+t,\; T\leq t\leq T_1\right\})}^{\frac{N+2}{N-2}}.
\end{equation*}
By finite speed of propagation, $u(t,x)=\tilde{u}(t,x)$ if $t\geq T$, $|x|> A+t$, and thus
\begin{equation*}
\left|\left\|u\right\|_{S(\left\{|x|\geq A+t,\; T\leq t\leq T_1\right\})}-\left\|S_{\lin}(\cdot-T)\vec{u}(T)\right\|_{S(\left\{|x|\geq A+t,\; T\leq t\leq T_1\right\})}\right|
\leq C\left\|u\right\|_{S(\left\{|x|\geq A+t,\; T\leq t\leq T_1\right\})}^{\frac{N+2}{N-2}},
\end{equation*}
Assuming $\delta_1$ small, point \eqref{I:NL_L} follows immediately. An easy bootstrap argument yields point \eqref{I:L_NL}.
\end{proof}
\subsection{Profile decomposition}
\label{SS:profiles}
We finally recall that any sequence $\big\{(u_{0,n},u_{1,n})\big\}_n$ bounded in $\hdot\times L^2$ has a subsequence (still denoted by $\big\{(u_{0,n},u_{1,n})\big\}_n$) that admits a \emph{profile decomposition} \profiles, where for all $j$, $U_{\lin}^j$ is a solution of the linear wave equation \eqref{E3} and for all $j,n$, $\lambda_{j,n}>0$, $x_{j,n}\in \RR^N$ and $t_{j,n}\in \RR$, have the following properties:
$$j\neq k\Longrightarrow \lim_{n\to\infty} \frac{\lambda_{j,n}}{\lambda_{k,n}}+\frac{|x_{j,n}-x_{k,n}|}{\lambda_{j,n}}+\frac{|t_{j,n}-t_{k,n}|}{\lambda_{j,n}}=+\infty$$
(pseudo-orthogonality) and
$$\lim_{n\to\infty}\limsup_{n\to\infty} \left\|w_n^J\right\|_{S(\RR)}=0,$$
where 
\begin{equation}
\label{wnJ}
w_n^J(t)=S_{\lin}(t)(u_{0,n},u_{1,n})-\sum_{j=1}^J U_{\lin,n}^j(t), 
\end{equation} 
and
$$ U_{\lin,n}^j(t,x)=\frac{1}{\lambda_{j,n}^{\frac{N-2}{2}}}U^j\left( \frac{t-t_{j,n}}{\lambda_{j,n}},\frac{x-x_{j,n}}{\lambda_{j,n}} \right).$$
The existence of the profile decomposition was established in \cite{BaGe99} for $N=3$ (see \cite{Bulut10} for higher dimensions). We refer to \cite{BaGe99} for the properties of this profile decomposition (see also \cite[Section 3]{DuKeMe15Pb} for a review).
\section{Scattering to a linear solution outside a large wave cone}
\label{S:large_wave_cone}
We let $\tau_n\to+\infty$ and (after extraction), 
\begin{align}
\label{E12}
(v_0,v_1)&=w-\lim_{n\to\infty} \vec{S}_{\lin}(-\tau_n)\vec{u}(\tau_n)\text{ in }\hdot\times L^2\\
\label{E13}
v_{\lin}(t)&=S_{\lin}(t)(v_0,v_1),
\end{align} 
where $w-\lim$ stands for the weak limit.
\begin{prop}
\label{P:E5}
 Let $u$ be as in Theorem \ref{T:E1}, and assume that the conclusion of Theorem \ref{T:E1} does not hold. Then there exists $\overline{A}\in \RR$ with the following properties:
 \begin{gather}
  \label{E14}
  \forall A>\overline{A},\quad \left\|u\right\|_{S(\{t>0,\,|x|>t+A\})}<\infty\\
  \label{E14'}
  \forall A>\overline{A},\quad \lim_{t\to\infty}\left\||x|^{-1}u(t)\right\|_{L^2(\{|x|>t+A\})}+\left\|u(t)\right\|_{L^{\frac{2N}{N-2}}(\{|x|>t+A\})}=0\\
  \label{E15}
  \forall A>\overline{A},\quad \lim_{t\to\infty}\left\|\nabla_{t,x}(u-v_{\lin})(t)\right\|_{L^2(\{|x|>t+A\})}=0\\
  \label{E16}
  \|u\|_{S\left(\{t>0,|x|>t+\overline{A}\}\right)}=\infty.
 \end{gather}
\end{prop}
\begin{proof}
 \EMPH{Step 1}
 We note that $\|u\|_{S\left\{|x|>|t|+A\right\}}$ is finite for large $A>0$. Indeed, since 
 $$\left\|S_{\lin}(\cdot)(u_0,u_1)\right\|_{S(\RR)}<\infty,$$
 we have, for large $A$,
$$\left\|S_{\lin}(\cdot)(u_0,u_1)\right\|_{S(\{|x|\geq |t|+A\})}<\delta_1,$$
where $\delta_1$ is given by Claim \ref{Cl:FSP}. The conclusion follows from Claim \ref{Cl:FSP}.

\EMPH{Step 2} We let $\overline{A}\in \RR\cup\{-\infty\}$ be defined by
\begin{equation}
 \label{E20} 
 \overline{A}:=\inf\left\{A\in \RR\;:\; \|u\|_{S(\{t>0,\;|x|>t+A\})}<\infty\right\}.
\end{equation} 
In particular \eqref{E14} holds. We prove here that \eqref{E15} holds. Let $A>\overline{A}$. Let $v$ be the solution of 
\begin{equation}
 \label{E21}
 \left\{\begin{aligned}
       \partial_t^2v-\Delta v&=|u|^{\frac{4}{N-2}}u\indic_{|x|>t+A}\\  
       \vec{v}_{\restriction t=0}&=(u_0,u_1).
        \end{aligned}
\right.
\end{equation} 
 By the definition of $\overline{A}$, we see that the right-hand side of the first equation in \eqref{E21} is in $L^1([0,+\infty),L^2(\RR^N))$. By Claim \ref{Cl:E4}, there exists $v_{\lin}^A$, solution of \eqref{E3}, such that 
 \begin{equation}
  \label{E22}
  \lim_{t\to+\infty} \int \left|\nabla_{t,x}(v-v_{\lin}^A)(t,x)\right|^2\,dx=0.
 \end{equation} 
 By finite speed of propagation, $u(t,x)=v(t,x)$ for $(t,x)$ such that $t>0$, $|x|>A+t$, and thus
\begin{equation}
  \label{E23}
  \lim_{t\to+\infty} \int_{|x|>t+A} \left|\nabla_{t,x}(u-v_{\lin}^A)(t,x)\right|^2\,dx=0
 \end{equation} 
 and 
 \begin{equation*}
  \lim_{t\to+\infty}  \int_{|x|>t+A}\frac{1}{|x|^2} |u(t,x)|^2+|u(t,x)|^{\frac{2N}{N-2}}\,dx=0.
 \end{equation*} 
 It remains to prove 
 $$\lim_{t\to+\infty}\left\|\nabla_{t,x}\left(v_{\lin}^A(t)-v_{\lin}(t)\right)\right\|_{L^2(|x|>t+A)}=0.$$ 
 We let $G, G^A \in L^2(\RR\times S^{N-1})$ be the radiation fields associated to $v_{\lin}$ and $v_{\lin}^A$ respectively (see Theorem \ref{T:B2}). We will prove
 \begin{equation}
  \label{E26}
  \forall \omega\in S^{N-1},\; \forall \eta>A,\quad G(\eta,\omega)=G^A(\eta,\omega),
 \end{equation} 
 which, in view of \eqref{E23} and Theorem \ref{T:B2}, will yield the conclusion of Step 2.
 
 Fix $\eps>0$ and $\Phi\in L^2(\RR\times S^{N-1})$ such that $\Phi(\eta,\omega)=0$ if $\eta\leq A+\eps$. Let $w_{\lin}(t)=S_{\lin}(t)(w_0,w_1)$ be the solution of the linear wave equation \eqref{E3} whose associated radiation field for $t\to+\infty$ is $\Phi$ (see Theorem \ref{T:B2}). In other words,
 \begin{align}
  \label{E27}
  \lim_{t\to+\infty}\int_0^{+\infty}\int_{S^{N-1}} \left|r^{\frac{N-1}{2}}\partial_{t}w_{\lin}(t,r\omega)-\Phi(r-t,\omega)\right|^2\,d\omega\,dr=0\\
  \label{E27'}
 \lim_{t\to+\infty}\int_0^{+\infty}\int_{S^{N-1}} \left|r^{\frac{N-1}{2}}\partial_{r}w_{\lin}(t,r\omega)+\Phi(r-t,\omega)\right|^2\,d\omega\,dr=0.
  \end{align} 
  On one hand, we have
  \begin{multline}
   \label{E28}
   \Big(\vec{u}(\tau_n),S_{\lin}(\tau_n)(w_0,w_1)\Big)_{\hdot\times L^2}=\Big(S_{\lin}(-\tau_n)\vec{u}(\tau_n),(w_0,w_1)\Big)_{\hdot\times L^2}
   \\
   \underset{n\to\infty}{\longrightarrow} \left((v_0,v_1),(w_0,w_1)\right)_{\hdot\times L^2},
  \end{multline} 
  by the definition \eqref{E12} of $(v_0,v_1)$, and thus, using the isometry property of radiation fields:
  \begin{equation}
   \label{E29}
   \lim_{n\to\infty}
   \left(\vec{u}(\tau_n),S_{\lin}(\tau_n)(w_0,w_1)\right)_{\hdot\times L^2}=2\int_{-\infty}^{+\infty}\int_{S^{N-1}} G(\eta,\omega)\,\Phi(\eta,\omega)\,d\omega\,d\eta.
  \end{equation} 
  On the other hand:
  \begin{multline*}
   \left(\vec{u}(\tau_n),S_{\lin}(\tau_n)(w_0,w_1)\right)_{\hdot\times L^2}= \int\nabla_{t,x}u(\tau_n,x)\cdot\nabla_{t,x} w_{\lin}(\tau_n,x)\,dx\\
   =\int_0^{+\infty} \int_{S^{N-1}}r^{\frac{N-1}{2}}\left(\partial_tu(\tau_n,r\omega)-\partial_ru(\tau_n,r\omega)\right)\Phi(r-\tau_n,\omega)\,d\omega\,dr+o_n(1)\\
   =2\int_0^{+\infty} \int_{S^{N-1}}G^A(r-\tau_n,\omega)\Phi(r-\tau_n,\omega)\,d\omega\,dr+o_n(1),
  \end{multline*}
  where at the last line we used that $\Phi(r-\tau_n,\omega)=0$ if $r-\tau_n\leq A$, \eqref{E23}, and the definition of $G^A$. Hence
  $$\lim_{n\to\infty} \left(\vec{u}(\tau_n),\vec{S}_{\lin}(\tau_n)(w_0,w_1)\right)_{\hdot\times L^2}=2\int_{-\infty}^{+\infty}\int_{S^{N-1}} G^A(\eta,\omega)\Phi(\eta,\omega)\,d\omega\,d\eta.$$
  Combining with \eqref{E29}, we obtain that for all $\Phi\in L^2(\RR\times S^{N-1})$ such that $\Phi(\eta)=0$ if $\eta\leq A+\eps$,
  $$\int_{-\infty}^{+\infty} (G^A-G)\Phi=0.$$
  Using this equality with 
  $$ \Phi(\eta,\omega)=(G^A(\eta,\omega)-G(\eta,\omega))\indic_{\eta\geq A+\eps}$$
  yields \eqref{E26} (since $\eps>0$ can be taken arbitrarily small), which concludes Step 2.
  
  \EMPH{Step 3} By Step 2, and since we are assuming that the conclusion of Theorem \ref{T:E1} does not hold, $\overline{A}\in \RR$. In this step we prove:
  \begin{equation}
   \label{E31}
   \|u\|_{S\left(\{t>0,\; |x|>\overline{A}+t\}\right)}=\infty,
  \end{equation} 
  which will conclude the proof of Proposition \ref{P:E5}. We argue by contradiction, assuming
  \begin{equation}
   \label{E32} \|u\|_{S\left(\{ t>0,\;|x|\geq \overline{A}+t\}\right)}<\infty.
  \end{equation} 
  Let $\delta_1$ be given by Claim \ref{Cl:FSP} and $T\gg 1$ such that 
  $$\|u\|_{S\left(\{t>T,\; |x|\geq \overline{A}+t\}\right)}<\delta_1/4.$$
  Then by Claim \ref{Cl:FSP},
  $$\left\|S_{\lin}(t-T)\vec{u}(T)\right\|_{S\left(\{t>T,\;|x|>\overline{A}+t\}\right)}<\delta_1/2.$$
  Let $\eps>0$ such that 
  \begin{equation}
   \label{E33}
   \left\|S_{\lin}(t-T)\vec{u}(T)\right\|_{S\left(\{t>T,\;|x|>\overline{A}-\eps+t\}\right)}<\delta_1.
  \end{equation} 
  Then, again by Claim \ref{Cl:FSP},
  $$\|u\|_{S\left(\{t>T,\;|x|>\overline{A}-\eps+t\}\right)}<2\delta_1,$$
  which contradicts the definition of $\overline{A}$, concluding the proof.
 \end{proof}
\section{Singular and regular directions}
\label{S:sing_reg}
In this section, we still assume that $u$ is a solution of \eqref{NLW} that satisfies the assumptions of Theorem \ref{T:E1} and not its conclusion. We let $\overline{A}\in \RR$ be defined by Proposition \ref{P:E5} in the preceding section.
\begin{defi}
 \label{D:E6}
 The set $\RRR$ of \emph{regular directions} is the set of $\omega\in S^{N-1}$ such that there exists $\eps>0$ with 
 \begin{equation}
  \label{E:34}
  \left\|u\right\|_{S\left( \left\{t>0,\; |x|>\overline{A}-\eps+t\text{ and }\left|\widehat{\left(\omega,x\right)}\right|<\eps \right\}\right)}<\infty,
 \end{equation} 
 where $\widehat{(\omega,x)}\in [-\pi,\pi)$ is the angle between $\omega$ and $x$. The set $\SSS$ of \emph{singular directions} is defined as $\SSS:=S^{N-1}\setminus \RRR$.
\end{defi}
In this section we prove:
\begin{prop}
 \label{P:E7}
 Under the above assumptions, the set $\SSS$ is finite and nonempty.
\end{prop}
\begin{prop}
\label{P:E8}
 There exists $\delta_2>0$ such that if $\omega\in S^{N-1}$ satisfies, for some $\eps>0$ 
 \begin{equation}
  \label{E35}
  \liminf_{T\to+\infty} \left\|S_{\lin}(\cdot-T)\vec{u}(T)\right\|_{S\left( \left\{ t>T,\; \overline{A}-\eps+t\leq |x|\leq \overline{A}+\eps+t\text{ and }\widehat{(\omega,x)}\leq \frac{4}{\sqrt{T}}\right\} \right)}<\delta_2,
 \end{equation} 
 then $\omega\in \RRR$.
\end{prop}
We first prove Proposition \ref{P:E7} assuming Proposition \ref{P:E8}.
\subsection{Finiteness of the set of singular directions}
We first prove by contradiction that $\SSS$ is non-empty.

Assume that $\SSS$ is empty. Thus $\RRR=S^{N-1}$ and for all $\omega$ belonging to $S^{N-1}$, there exists $\eps(\omega)$ such that 
\begin{equation}
 \label{E36}
  \left\|u\right\|_{S\left( \left\{t>0,\; |x|>\overline{A}-\eps(\omega)+t\text{ and }\left|\widehat{\left(\omega,x\right)}\right|<\eps (\omega)\right\}\right)}<\infty.
 \end{equation} 
 By the compactness of $S^{N-1}$, we can find $\omega_1,\ldots,\omega_J$ in $S^{N-1}$ such that 
 \begin{equation}
  \label{E37}
  S^{N-1}=\bigcup_{j=1}^J\left\{\omega\in S^{N-1}\;:\; \left|\widehat{(\omega,\omega_j)}\right|<\eps(\omega_j)\right\}.
 \end{equation} 
 Letting $\eps=\min_{j=1\ldots J} \eps(\omega_j)$, we see that 
 $$\|u\|_{S\left( \left\{t>0,\; |x|>\overline{A}-\eps+t \right\}\right)}<\infty,$$
 contradicting the definition of $\overline{A}$.
 
 We next prove that $\SSS$ is finite. Let $\omega_1$,\ldots,$\omega_J$ be \emph{two by two distinct} elements of $\SSS$. Then for large $T$,  the sets 
 $$\left\{(t,x)\;:\; t>T,\; \overline{A}-\eps+t\leq |x|\leq \overline{A}+\eps+t\text{ and }\left|\widehat{(\omega_j,x)}\right|\leq \frac{4}{\sqrt{T}}\right\}$$
 are pairwise disjoint.  As a consequence, for large $T$,
 \begin{multline}
  \label{E38}
  \sum_{j=1}^J \left\|S_{\lin}(\cdot-T)\vec{u}(T)\right\|_{L^{\frac{2(N+1)}{N-2}}\left(\left\{ t>T,\; \overline{A}-\eps+t\leq |x|\leq \overline{A}+\eps+t\text{ and }\left|\widehat{(\omega_j,x)}\right|\leq \frac{4}{\sqrt{T}}\right\}\right)}^{\frac{2(N+1)}{N-2}}\\
  \leq \left\|S_L(\cdot-T)\vec{u}(T)\right\|_{L^{\frac{2(N+1)}{N-2}}\left([\overline{A}-\eps+T,+\infty)\times \RR^N\right)}^{\frac{2(N+1)}{N-2}}.
 \end{multline}
Since $\overline{u}(T)$ is bounded in $\hdot\times L^2$, the right-hand side of \eqref{E38} is bounded independently of $T>0$. 

Let 
$$\Omega_j=\Big\{ (t,x)\;:\; t>T,\; \overline{A}-\eps+t\leq |x|\leq \overline{A}+\eps+t\text{ and } \left|\widehat{(\omega_j,x)}\right|\leq 4/\sqrt{T}\Big\}.$$
We will prove that $\SSS$ is finite distinguishing between $N=3$ and $N=4,5$. If $N=3$, we have
\begin{equation}
 \label{EC1}
 \left\|S_{\lin}(\cdot-T)\vec{u}(T)\right\|_{S(\Omega_j)}\leq \left\|S_{\lin}(\cdot-T)\vec{u}(T)\right\|_{L^4_tL^{12}_x(\Omega_j)}^{\frac{3}{5}}\left\|S_{\lin}(\cdot-T)\vec{u}(T)\right\|^{\frac{2}{5}}_{L^8(\Omega_j)}
\end{equation} 
Since $\omega_j\in \SSS$, the left-hand side of \eqref{EC1} is, according to Proposition \ref{P:E8}, bounded from below by $\frac{\delta_2}{2}$ for large $n$. Combining with the boundedness of $\vec{u}(t)$ in $\hdot\times L^2$ and Strichartz estimates, we deduce that for large $t$,$$\delta_2\leq C\left\|S_{\lin}(\cdot-T)\vec{u}(T)\right\|^{\frac{2}{5}}_{L^8(\Omega_j)}.$$
Hence by \eqref{E38},
$$ J\delta_2^{20}\leq C \left\|S_{\lin}(\cdot-T)\vec{u}(T)\right\|^8_{L^8\left([\overline{A}-\eps+T,+\infty)\times \RR^N\right)}\leq C,$$
where the right-hand side inequality is a consequence of Strichartz inequality and the boundedness of $\vec{u}$ in $\hdot\times L^2$. This proves that $\SSS$ is finite.

If $N=4$ or $N=5$, the proof is very close, using 
\begin{equation}
 \label{EC2}
 \left\|S_{\lin}(\cdot-T)\right\|_{S(\Omega_j)} \leq C\left\|S_{\lin}(\cdot-T)\right\|^{1-\theta}_{L^2_tL^{\frac{2N}{N-3}}(\Omega_j)}\left\|S_{\lin}(\cdot-T)\vec{u}(T)\right\|^{\theta}_{L^{\frac{2(N+1}{N-2}}(\Omega_j)},
\end{equation} 
$\theta=\frac{(N+1)(6-N)}{3(N+2)}$,
instead of \eqref{EC1}. We omit the details. The proof of Proposition \ref{P:E7} is complete.\qed
\subsection{A geometrical lemma}
We now turn to some elementary geometrical properties that will be useful in the proof of Proposition \ref{P:E8}. Without loss of generality, we will assume $\omega=e_1:=(1,0,\ldots,0)$.

For $\theta\in (0,\pi/2)$, we let 
$$\Gamma_{\theta}:=\Big\{ x\in \RR^N\setminus\{0\}\;:\;\left|\widehat{(e_1,x)}\right|\geq \frac{\pi}{2}+\theta\Big\}\cup \{0\},$$
where as before $\widehat{(e_1,x)}$ is the angle between $e_1$ and $x$. If $\tau>0$, we define:
$$D_{\tau,\theta}=\Big\{x\in  \RR^N\;:\;d(x,\Gamma_{\theta})>\tau\Big\},$$
where $d(x,\Gamma_{\theta})=\inf \left\{|y-x|\;:\;y\in \Gamma_{\theta}\right\}$ is the distance between $x$ and $\Gamma_{\theta}$
(see Figure \ref{Fig1}).
\begin{figure}[h]
\caption{}
\label{Fig1}
\includegraphics{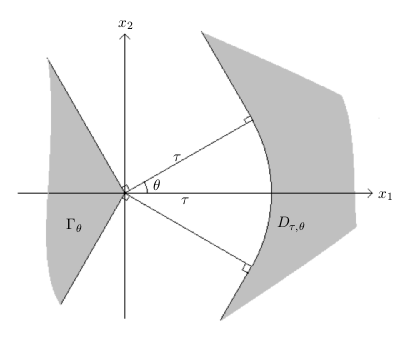} 
\end{figure}
Then:
\begin{lemma}
 \label{L:E8}
 \begin{enumerate}
  \item \label{LE8:1} $\ds D_{\tau,\theta}\subset\{|x|>\tau\}$.
  \item \label{LE8:2} $\Big( |x|>\tau\text{ and }\left|\widehat{(x,e_1)}\right|<\theta\Big)\Longrightarrow x\in D_{\tau,\theta}$.
  \item \label{LE8:3} If $\ell>0$, $x\in D_{\tau,\theta}$ and $|x|\leq \tau+\ell$, then 
  $$\left|\widehat{(x,e_1)}\right|\leq \theta+\sqrt{\frac{\ell}{\tau+\ell}}.$$
 \end{enumerate}
\end{lemma}
\begin{proof}
 \EMPH{Proof of \eqref{LE8:1}} Since $0\in \Gamma_{\theta}$, $|x|\geq d(x,\Gamma_{\theta})$ and \eqref{LE8:1} follows from the definition of $D_{\tau,\theta}$.
 
 \EMPH{Proof of \eqref{LE8:2}} Let $x$ such that $|x|>\tau$ and $\left|\widehat{(x,e_1)}\right|<\theta$. Rotating around the axis $Oe_1$, we can assume $x=(x_1,x_2,0,\ldots,0)$. But then \eqref{LE8:2} is clear from Figure \ref{Fig1}.
 
 \EMPH{Proof of \eqref{LE8:3}} We will use the following elementary inequality:
 
 \begin{equation}
  \label{E39}
  \forall s\in [0,\pi/2],\quad 1-\cos s\geq s^2/4.
 \end{equation} 
 Let $x\in D_{\tau,\theta}$ such that $|x|<\tau+\ell$. As before, we can assume $x=(x_1,x_2,0,\ldots,0)$. Denote by $\CCC(R)$ the circle of radius $R>0$ centered at the origin.
 
 \begin{figure}
 \caption{}
 \label{Fig2}
\includegraphics{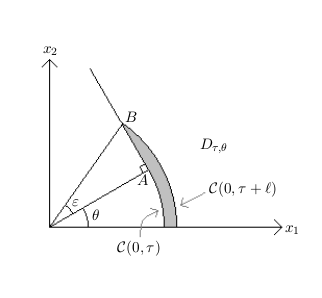} 
\end{figure}
 
 Let $A=(\tau \cos\theta,\tau\sin\theta,0,\ldots,0)$. The tangent at $A$ to the circle $\CCC(0,\tau)$ intersects the circle $\CCC(0,\tau+\ell)$ at a point $B$. Let $\eps$ be the angle $\widehat{(\vec{0A},\vec{0B})}$.
 
 From the conditions $|x|<\tau+\ell$ and $x\in D_{\tau,\theta}$, we get $\left|\widehat{(x,e_1)}\right|\leq \eps+\theta$. Indeed, in Figure \ref{Fig2}, $x$ must be in the dark region. 
 
 We have $\cos\eps=\tau/(\tau+\ell)$ and hence, by \eqref{E39},
 $$\frac{\tau}{\tau+\ell}\leq 1-\frac{\eps^2}{4}$$
 i.e.  $\eps^2/4\leq \ell/(\tau+\ell)$, which yields the conclusion of \eqref{LE8:3}.
 
 \subsection{Sufficient condition to be a regular direction}
 
 In this subsection we prove Proposition \ref{P:E8}. We assume that \eqref{E35} holds with $\omega=e_1$. Thus there exist $\eps>0$ and a sequence $\{t_n\}_n\to+\infty$ such that 
 \begin{equation}
  \label{E40}
  \lim_{n\to\infty}
  \left\|S_{\lin}(t-t_n)\vec{u}(t_n)\right\|_{S\left(\left\{t>t_n,\; \overline{A}-\eps+t\leq |x|\leq \overline{A}+\eps+t,\;\left|\widehat{(e_1,x)}\right|\leq 4/\sqrt{t_n}\right\}\right)}<\delta_2.
 \end{equation} 
 We denote
 \begin{equation}
  \label{E41}
  \DDD_n=\bigcup_{t>t_n} \{t\}\times D_{\overline{A}-\eps+t,\frac{1}{\sqrt{t_n}}},
 \end{equation} 
 where $\eps>0$ is as in \eqref{E40}. We claim that $\DDD_n$ satisfies the following causality property:
\begin{claim}
\label{Cl:E9}
 Let $(t,x)\in \DDD_n$ and $(t',x')\in \RR\times \RR^N$ with $t_n<t'<t$ and $|x'-x|<|t-t'|$. Then $(t',x')\in \DDD_n$.
\end{claim}
\begin{proof}
Indeed we must check that $x'\in D_{\overline{A}-\eps+t',\frac{1}{\sqrt{t_n}}}$, i.e. that $d\left(x',\Gamma_{\frac{1}{\sqrt{t_n}}}\right)>\overline{A}-\eps+t'$. By the triangle inequality,
\begin{equation*}
  d\left(x',\Gamma_{\frac{1}{\sqrt{t_n}}}\right)>d\left(x,\Gamma_{\frac{1}{\sqrt{t_n}}}\right)-|x-x'|>d\left(x,\Gamma_{\frac{1}{\sqrt{t_n}}}\right)-|t-t'|.
\end{equation*} 
Since $x\in \DDD_{\overline{A}-\eps+t,\frac{1}{\sqrt{t_n}}}$, we obtain
 $$d\left(x',\Gamma_{\frac{1}{\sqrt{t_n}}}\right)>\overline{A}-\eps+t-(t-t')=\overline{A}-\eps+t'.$$
 \end{proof}
 
 We divide the proof of Proposition \ref{P:E8} into two steps.
 
\EMPH{Step 1} We prove
\begin{equation}
 \label{E42}
 \limsup_{n\to\infty} \left\|S_{\lin}(\cdot-t_n)\vec{u}(t_n)\right\|_{S(\DDD_n)}<2\delta_2.
\end{equation} 
 Indeed, we first note that by the definition of $\overline{A}$, 
 $$\left\|u\right\|_{S\left\{t\geq 0,\; |x|\geq \overline{A}+\eps+t\right\}}<\infty.$$
 Thus 
 \begin{equation}
 \label{E43}
 \lim_{n\to\infty}\left\|u\right\|_{S\left\{t>t_n,\; |x|\geq \overline{A}+\eps+t\right\}}=0.
 \end{equation} 
 By Claim \ref{Cl:FSP},
 \begin{equation}
  \label{E44}
 \lim_{n\to\infty}\left\|S_L(\cdot-t_n)\vec{u}(t_n)\right\|_{S\left\{ t>t_n,\; |x|\geq \overline{A}+\eps+t\right\}}=0.
 \end{equation} 
 We are thus reduced to prove
 \begin{equation}
  \label{E45}
 \limsup_{n\to\infty}\left\|S_L(\cdot-t_n)\vec{u}(t_n)\right\|_{S\left\{(t,x)\in \DDD_n\;:\;|x|< \overline{A}+\eps+t\right\}}<2\delta_2.
 \end{equation} 
 By Lemma \ref{L:E8} \eqref{LE8:3}, if $t>t_n$ and $x\in D_{\overline{A}-\eps+t,\frac{1}{\sqrt{t_n}}}$ satisfies $|x|<\overline{A}+\eps+t$, then 
 $$\left|\widehat{(x,e_1)}\right|\leq \frac{1}{\sqrt{t_n}}+2\sqrt{\frac{2\eps}{\overline{A}+\eps+t}},$$
 so that, if $\eps>0$ is small enough and $n$ large, 
 $$\left|\widehat{(x,e_1)}\right|\leq \frac{2}{\sqrt{t_n}}.$$
 As a consequence
 $$\Big( (t,x)\in \DDD_n \text{ and }|x|<\overline{A}+\eps+t\Big)\Longrightarrow \Big(t>t_n,\; \overline{A}-\eps+t<|x|<\overline{A}+\eps+t\text{ and } \left|\widehat{(x,e_1)}\right|\leq \frac{2}{\sqrt{t_n}}\Big),$$
 and \eqref{E45} follows from \eqref{E40}.
 
 \EMPH{Step 2} We prove that for large $n$, $\|u\|_{S(\DDD_n)}$ is finite. More precisely
 \begin{equation}
  \label{E46}
  \limsup_{n\to\infty} \|u\|_{S(\DDD_n)}\leq 3\delta_2.
 \end{equation} 
 Since, by Lemma \ref{L:E8}, \eqref{LE8:2},
 $$ \left(t>t_n,\; |x|>\overline{A}-\eps+t\text{ and }\left|\widehat{(x,e_1)}\right|\leq \frac{1}{\sqrt{t_n}}\right)\Longrightarrow (t,x)\in \DDD_n,$$
 we see that \eqref{E46} implies $e_1\in \RRR$ as desired (see Definition \ref{D:E6}).

 For any $T>t_n$, we let $\DDD_{n,T}=\left\{(t,x)\in \DDD_n\; :\; t\in [t_n,T]\right\}$. We will prove that for large $n$ $\|u\|_{S(\DDD_{n,T})}\leq 3\delta_2$ by a bootstrap argument. By finite speed of propagation, Claim \ref{Cl:E9} and Strichartz estimates,
 \begin{equation*}
  \forall T>t_n,\quad 
  \|u\|_{S(\DDD_{n,T})}\leq \left\|S_{\lin}(\cdot-t_n)\vec{u}(t_n)\right\|_{S(\DDD_{n,T})}+C\|u\|^{\frac{N+2}{N-2}}_{S(\DDD_{n,T})}.
 \end{equation*} 
 Thus for large $n$, in view of Step 1, 
 $$\|u\|_{S(\DDD_{n,T})}\leq \frac 52\delta_2+C\|u\|^{\frac{N+2}{N-2}}_{S(\DDD_{n,T})},$$
 and the result follows if $\delta_2$ is small enough.
\end{proof}
\section{Conditions on the profiles}
\label{S:cond_profiles}
In this part we will prove the following two lemmas:
\begin{lemma} 
 \label{L:E10}
 Let $\{t_n\}_n\to+\infty$ be such that $\vec{u}(t_n)$ has a profile decomposition. Then there is no profile such that $U^j_{\lin}\not\equiv 0$ and the following hold:
 \begin{gather}
 \label{E47}
  \lim_{n\to\infty}|x_{j,n}|-t_n>\overline{A}\\
 \label{E48}
 \lim_{n\to\infty}|x_{j,n}|-t_n-|t_{j,n}|\geq \overline{A}\text{ and }\\
 \label{E49}
 \lim_{n\to\infty} \lambda_{j,n}=0.
 \end{gather}
\end{lemma}
\begin{lemma}
 \label{L:E11}
 There exists $\delta_3>0$ with the following property. Let $\{t_n\}\to +\infty$ and $\omega\in \SSS$. Then there exists a subsequence of $\{t_n\}_n$ (that we still denote by $\{t_n\}_n$) such that $\{u(t_n)\}_n$ has a profile decomposition \profiles such that 
 \begin{gather}
  \label{E68} \lim_{n\to\infty}x_{1,n}-(\overline{A}+t_n)\omega=0\\
  \label{E69} \lim_{n\to\infty}\lambda_{1,n}=0\\
  \label{E70} \lim_{n\to\infty} t_{1,n}=0 \text{ and}\\
  \label{E71} \left\|U^1_{\lin}\right\|_{S(\RR)}\geq \delta_3.
 \end{gather}
\end{lemma}

\begin{proof}[Proof of Lemma \ref{L:E10}]
 We argue by contradiction. Let $j\geq 1$ be such that $U^j_{\lin}\not\equiv 0$, and \eqref{E47}, \eqref{E48} and \eqref{E49} hold. Using that 
 \begin{align*}
  j\neq k&\Longrightarrow \lim_{n\to\infty}\int\nabla_{t,x}U^{j}_{\lin,n}(0,x)\cdot\nabla_{t,x}U^k_{\lin,n}(0,x)\,dx=0\\
  j\leq J&\Longrightarrow \lim_{n\to\infty}\int\nabla_{t,x}U^{j}_{\lin,n}(0,x)\cdot\nabla_{t,x}w^J_{\lin,n}(0,x)\,dx=0,
 \end{align*}
we obtain
\begin{equation}
 \label{E50}
 \lim_{n\to\infty} \int \nabla_{t,x}u(t_n,x)\cdot\nabla_{t,x}U_{\lin,n}^j(0,x)\,dx=\int|\nabla_{t,x}U^j_{\lin}(0,x)|^2\,dx>0.
\end{equation} 
We will reach a contradiction from \eqref{E50}, using the localization properties of $U_{\lin,n}^j$.

\EMPH{Case 1} Assume that $\{-t_{j,n}/\lambda_{j,n}\}_n$ is bounded. Extracting subsequences and time translating $U^j_{\lin}$ if necessary, we can assume $t_{j,n}=0$ for all $n$. By \eqref{E47}, there exists $\eps>0$ such that 
\begin{equation}
 \label{E51} |x_{j,n}|-t_n\geq \overline{A}+2\eps
\end{equation} 
for large $n$. Since
\begin{equation}
 \label{E52}
 \nabla_{t,x}U^j_{\lin,n}=\frac{1}{\lambda_{j,n}^{\frac{N}{2}}}\nabla_{t,x}U^j_{\lin}\left( 0,\frac{x-x_{j,n}}{\lambda_{j,n}} \right)
\end{equation} 
we obtain, using also assumption \eqref{E49},
\begin{multline*}
 \int \nabla_{t,x}u(t_n,x)\cdot\nabla_{t,x}U_{\lin,n}^j(0,x)\,dx\\
 =\int_{|x|>\overline{A}+\eps+t_n} \nabla_{t,x}u(t_n,x)\cdot\nabla_{t,x}U_{\lin,n}^j(0,x)\,dx+o(1)\\
=\int_{|x|>\overline{A}+\eps+t_n} \nabla_{t,x}v_{\lin}(t_n,x)\cdot\nabla_{t,x}U_{\lin,n}^j(0,x)\,dx+o(1)\\
=\int \nabla_{t,x}v_{\lin}(t_n,x)\cdot\nabla_{t,x}U_{\lin,n}^j(0,x)\,dx+o(1),
 \end{multline*}
as $n\to\infty$. We have used \eqref{E15} in Proposition \ref{P:E5} and, at the last line, \eqref{E49}, \eqref{E51} and \eqref{E52}. As a consequence, using the asymptotic behaviour of $v_{\lin}(t)$ as $t\to+\infty$ (see Theorem \ref{T:B2}), we obtain
$$\lim_{n\to\infty}\int \nabla_{t,x}u(t_n,x)\cdot\nabla_{t,x}U_{\lin,n}^j(0,x)\,dx=0,$$
contradicting \eqref{E50}.

\EMPH{Case 2} We assume (after extraction),
\begin{equation}
 \label{E53}
 \lim_{n\to\infty}\frac{-t_{j,n}}{\lambda_{j,n}}=-\infty
\end{equation} 
(the proof in the case where this limit is $+\infty$ is the same). Let $G^j$ be the radiation field associated to $t\mapsto U^j_{\lin}(-t)$ (see Theorem \ref{T:B2}). Then
\begin{equation}
 \label{E54}
 \lim_{t\to-\infty} \int_0^{+\infty} \int_{S^{N-1}} \left| r^{N-1}|\nabla_{t,x} U^j_{\lin}(t,r\omega)|^2-2\left|G^j(r+t,\omega)\right|^2\right|\,d\omega\,dr=0.
\end{equation} 

\EMPH{Subcase 2a} We assume, in addition to \eqref{E53},
\begin{equation}
 \label{E56}
 \lim_{n\to\infty} |x_{j,n}|-t_n-|t_{j,n}|>\overline{A},
\end{equation} 
i.e. that \eqref{E48} holds with a strict inequality. Fix a small $\eps>0$, and $R\gg 1$ such that 
\begin{equation}
 \label{E57}
 \int_{|\eta|\geq R}\int_{S^{N-1}} \left|G^j(\eta,\omega)\right|^2\,d\omega\,d\eta<\eps.
\end{equation} 
Using \eqref{E54} and \eqref{E57}, we obtain that for large $n$,
\begin{equation}
 \label{E58}
 \int_{\big||x_{j,n}-x|-|t_{j,n}|\big|\geq R\lambda_{j,n}} \left|\nabla_{t,x}U^j_{\lin,n}(0,x)\right|^2,dx <C\eps. 
\end{equation} 
Hence for large $n$, 
\begin{equation}
 \label{E59}
 \int \nabla_{t,x}u(t_n)\cdot\nabla_{t,x}U^{j}_{\lin,n}(0,x)\,dx\leq \int_{\big||x_{j,n}-x|-|t_{j,n}|\big|\leq R\lambda_{j,n}} \nabla_{t,x}u(t_n)\cdot\nabla_{t,x}U^{j}_{\lin,n}(0,x)\,dx+C\sqrt{\eps}
\end{equation} 
(where the constant $C>0$ might depend on $U^j$, $\sup_{t\geq 0} \|\nabla_{t,x}u\|_{\hdot\times L^2}$ but is of course independent on $n$). By the triangle inequality, \eqref{E49} and \eqref{E56}, there exists $\eps'>0$ such that for large $n$,
\begin{equation}
 \label{E60}
 |x_{j,n}-x|-|t_{j,n}|\leq R\lambda_{j,n}\Longrightarrow |x|\geq |x_{j,n}|-|t_{j,n}|-R\lambda_{j,n}\Longrightarrow |x|\geq t_n+\overline{A}+\eps'.
\end{equation} 
Using \eqref{E15} in Proposition \ref{P:E5} and \eqref{E59}, we obtain that for large $n$,
\begin{multline*}
 \int\nabla_{t,x}u(t_n,x)\cdot\nabla_{t,x}U^j_{\lin,n}(0,x)\,dx\leq \int_{\left| |x_{j,n}-x|-|t_{j,n}|\right|\leq R\lambda_{j,n}} \nabla_{t,x}v_{\lin}(t_n,x)\cdot\nabla_{t,x}U^j_{\lin,n}(0,x)\,dx+C\sqrt{\eps}\\
 \leq \int \nabla_{t,x}v_{\lin}(t_n,x)\cdot\nabla_{t,x}U^j_{\lin,n}(0,x)\,dx+C\sqrt{\eps}.
\end{multline*}
Since 
$$\lim_{n\to\infty}\int \nabla_{t,x} v_{\lin}(t_n,x)\cdot\nabla_{t,x}U^j_{\lin,n}(0,x)\,dx=0,$$ 
we deduce
(using that $\eps>0$ is arbitrarily small)
\begin{equation}
 \label{E61}
\limsup_{n\to\infty}\int \nabla_{t,x} u(t_n,x)\cdot\nabla_{t,x}U^j_{\lin,n}(0,x)\,dx\leq 0, 
\end{equation} 
which contradicts \eqref{E50}.

\EMPH{Subcase 2b} We assume, in addition to \eqref{E53},
\begin{equation}
 \label{E62}
 \lim_{n\to\infty} |x_{j,n}|-t_n-|t_{j,n}|=\overline{A},
\end{equation}
i.e. that \eqref{E48} holds with an equality. In view of \eqref{E47}, we must have
\begin{equation}
 \label{E63}
 \lim_{n\to\infty}|t_{j,n}|>0.
\end{equation} 
Extracting subsequences, we can assume
\begin{equation}
 \label{E64}
 \lim_{n\to\infty} \frac{x_{j,n}}{|x_{j,n}|}=\omega_0\in S^{N-1}.
\end{equation} 
Let $\eps>0$. Choose $R\gg 1$ such that \eqref{E57} holds, and $\alpha>0$ such that
\begin{equation}
 \label{E65}
 \int_{\RR}\int_{\substack{\omega\in S^{N-1}\\ |\omega+\omega_0|\leq \alpha}} \left|G^j(\eta,\omega)\right|^2\,d\omega\,d\eta<\eps.
\end{equation} 
In view of \eqref{E57} and \eqref{E65}, we have, for large $n$,
\begin{equation}
 \label{E66}
 \int_{\big||x-x_{j,n}|-|t_{j,n}|\big|\geq R\lambda_{j,n}} \left|\nabla_{t,x}U^j_{\lin,n}(0,x)\right|^2\,dx + 
\int_{\left| \frac{x-x_{j,n}}{|x-x_{j,n}|}+\omega_0\right|\leq \alpha} \left|\nabla_{t,x}U^j_{\lin,n}(0,x)\right|^2\,dx\leq C\eps,
 \end{equation} 
 and thus, for large $n$ again,
 $$\int\nabla_{t,x}u(t_n,x)\cdot\nabla_{t,x}U^{j}_{\lin,n}(0,x)\,dx\leq \int_{Z_n}\nabla_{t,x}u(t_n,x)\cdot\nabla_{t,x}U^j_{\lin,n}(0,x)\,dx+C\sqrt{\eps}. $$ 
 where 
 $$Z_n:=\Big\{ x\in \RR^N\;:\; \Big||x-x_{j,n}|-|t_{j,n}|\Big|\leq R\lambda_{j,n}\text{ and } \left|\frac{x-x_{j,n}}{|x-x_{j,n}|}+\omega_0\right|\geq \alpha\Big\}.$$
 We next prove that there exists $\eps'>0$ such that for large $n$,
 \begin{equation}
  \label{E67}
  x\in Z_n\Longrightarrow |x|\geq \overline{A}+\eps+t_n.
 \end{equation} 
 Assuming \eqref{E67}, we can prove \eqref{E61} exactly as in subcase 2a, obtaining again a contradiction.
 
 We have $x=x_{j,n}+x-x_{j,n}$. 
 \begin{align*}
  x\cdot \frac{x_{j,n}}{|x_{j,n}|}&=|x_{j,n}|+(x-x_{j,n}) \cdot\frac{x_{j,n}}{|x_{j,n}|}\\
  &=|x_{j,n}|+|x-x_{j,n}| \frac{(x-x_{j,n})\cdot x_{j,n}}{|x-x_{j,n}||x_{j,n}|}\\
  &=|x_{j,n}|+|x-x_{j,n}| \left(\frac{x-x_{j,n}}{|x-x_{j,n}|}\omega_0 +o_n(1)\right).
 \end{align*}
Taking the square of the inequality
$\left|\frac{x-x_{j,n}}{|x-x_{j,n}|}+\omega_0\right|\geq \alpha$ and expanding, we obtain
$\omega_0\cdot \frac{x-x_{j,n}}{|x-x_{j,n}|}\geq \frac{\alpha^2}{2}-1$ for $x\in Z_n$. Hence, for large $n$, if $x\in Z_n$,
\begin{multline*}
 |x|\geq x\cdot \frac{x_{j,n}}{|x_{j,n}|}\geq |x_{j,n}|+\left(\frac{\alpha^2}{4}-1\right)\left|x-x_{j,n}\right|\\
 \geq \overline{A}+t_n+|t_{j,n}|+\left(\frac{\alpha^2}{4}-1\right)\big(|t_{j,n}|+R\lambda_{j,n}+o_n(1)\big) 
 \geq t_n+\overline{A}+\frac{\alpha^2}{8}|t_{j,n}|
\end{multline*}
for large $n$, using \eqref{E49} and \eqref{E63}. In view of \eqref{E63}, the desired conclusion \eqref{E67} follows. This concludes the proof of Lemma \ref{L:E10}.
\end{proof}
\begin{proof}[Proof of Lemma \ref{L:E11}]
Assume without loss of generality that $\omega=e_1$.

 Let (after extraction) \profiles be a profile decomposition of $\{\vec{u}(t_n)\}_n$. Let $\delta_3>0$ to be specified later and $J_1$ be such that for $1\leq j\leq J_1$, $\|U^j_{\lin}\|_{S(\RR)}\geq \delta_3$ and for $j>J_1$, $\|U^{j}_{\lin}\|_{S(\RR)}<\delta_3$. Since 
 $$\sum_{j\geq 1} \left\|\vec{U}^j_{\lin}(0)\right\|^2_{\hdot\times L^2}<\infty,$$
 it is easy to see, using Strichartz estimates, that such a (finite) $J_1$ exists. Note that if $\delta_3$ is small enough, $J_1\geq 1$, because otherwise $u$ scatters, which contradicts our assumption that $\overline{A}$ is finite. 

 Note that $\left(U_{\lin}^j,\{t_{j,n},x_{j,n},\lambda_{j,n}\}_{n}\right)_{j\geq 1+J_1}$ is a profile decomposition for the sequence of remainders $\left\{w_n^{J_1}\right\}_{n}$.
 By Claim \ref{Cl:A1} in the appendix, there exists $\theta>0$ such that 
 \begin{equation}
  \label{E71'}
  \limsup_{n\to\infty} \left\|w_n^{J_1}\right\|_{S(\RR)} \leq C\delta_3^{\theta} ,
 \end{equation} 
 where $C$ depends only on the bound
 $$M:=\limsup_{t\to\infty} \left\|\vec{u}(t)\right\|_{\hdot\times L^2}.$$
 We choose $\delta_3$ such that $C\delta_3^{\theta}<\delta_2/2$, where $\delta_2$ is given by Proposition \ref{P:E8}. Extracting subsequences, we can assume that the following limits exist for all $j\in \{1,\ldots,J_1\}$:
 \begin{gather}
 \label{E73}
 \lim_{n\to\infty} \lambda_{j,n}\in [0,+\infty]\\
 \label{E74}
 \lim_{n\to\infty} t_{j,n}\in \RR\cup \{\pm\infty\}\\
 \label{E74'}
  \lim_{n\to\infty} \left|x_{j,n}-(\overline{A}+t_n)e_1\right|\in [0,+\infty].
 \end{gather}
We argue by contradiction, assuming that for all $j\in \{1,\ldots,J_1\}$, one of the limits \eqref{E73}, \eqref{E74} or \eqref{E74'} is not $0$. 

By Proposition \ref{P:E8} and Claim \ref{Cl:A1} in the appendix, it is sufficient to prove that for all $j\in \{1,\ldots,J_1\}$ there exists $\eps>0$ such that 
\begin{equation}
 \label{E75}
 \limsup_{n\to\infty} \left|I_{n,\eps}^j\right|<\frac{\delta_2}{4}J_1,
\end{equation} 
where
$$I_{n,\eps}^j:=\left( \int_{t_n}^{+\infty} \left( \int_{\substack{\overline{A}-\eps+|t|\leq |x|\leq  \overline{A}+\eps+|t|\\ \left|\widehat{(e_1,x)}\right|\leq 4/\sqrt{t_n}}} \left|U^j_{\lin,n}(t-t_n,x)\right|^{\frac{2(N+2)}{N-2}}\,dx \right)^{1/2}\,dt \right)^{\frac{N-2}{N+2}}.$$
This would yield, by the triangle inequality, \eqref{E71'} and the bound $C\delta_3^{\theta}<\delta_2/2$,
$$\limsup_{n\to\infty} \left\| S_L(\cdot-t_n)\vec{u}(t_n)\right\|_{S\left(\left\{\overline{A}-\eps+|t|\leq |x|\leq \overline{A}+\eps+|t|,\;\left|\widehat{(e_1,x)}\right|\leq 4/\sqrt{t_n}\right\}\right)}<\delta_2,$$
and thus, by Proposition \ref{P:E8}, that $e_1$ is a regular point, a contradiction.
 
 By the change of variables 
 $$ s=\frac{t-t_n-t_{j,n}}{\lambda_{j,n}},\quad y=\frac{x-x_{j,n}}{\lambda_{j,n}},$$
 we obtain
 \begin{equation}
  \label{E76}
  I_{n,\eps}=\int_{-t_{j,n}/\lambda_{j,n}}^{+\infty} \left( \int_{Q_{n,\eps}(s)} \left|U^j_{\lin}(s,y)\right|^{\frac{2(N+2)}{N-2}}\,dy \right)^{1/2}\,ds
 \end{equation} 
 where $Q_{n,\eps}(s)$ is the set of $y\in \RR^N$ such that the absolute value of the angle between $\lambda_{j,n}y+x_{j,n}$ and $e_1$ is $\leq 4/\sqrt{t_n}$ and $$\overline{A}-\eps+\left|\lambda_{j,n}s+t_n+t_{j,n}\right|<|x_{j,n}+\lambda_{j,n}y|<\overline{A}+\eps+|\lambda_{j,n}s+t_n+t_{j,n}|.$$
 
 If $\lim_{n\to\infty} -t_{j,n}/\lambda_{j,n}=+\infty$, then we see that $\lim_{n\to\infty} I_{n,\eps}=0$ and \eqref{E75} follows. We are thus reduced to the case 
 \begin{equation}
  \label{E77}
  \lim_{n\to\infty} \frac{-t_{j,n}}{\lambda_{j,n}}=-\infty\quad\text{or}\quad \forall n,\; t_{j,n}=0.
 \end{equation} 
 We note that if for almost every $s\in \RR$,
 \begin{equation}
  \label{E78}
  \indic_{Q_{n,\eps}(s)}(y) \underset{n\to\infty}{\longrightarrow}0 \text{ for a.a. }y,
 \end{equation} 
 then $\lim_n I_{n,\eps}=0$ by dominated convergence and we are done. 
 
\EMPH{Case 1} If $\lim_{n\to\infty} \lambda_{j,n}=+\infty$, then $Q_{n,\eps}(s)$ is for all $s$ included in an annulus of length $\leq C/\lambda_{j,n}$ which proves that \eqref{E78} holds.

\EMPH{Case 2} If $\lim_{n\to\infty}\lambda_{j,n}=0$, we distinguish 3 subcases according to the limit
\begin{equation}
 \label{F77}
 \ell=\lim_{n\to\infty}|x_{j,n}|-(\overline{A}+t_n).
\end{equation} 
\begin{itemize}
 \item 
If $\ell>0$, then by Lemma \ref{L:E10} we must have 
\begin{equation}
 \label{F77'}
\lim_{n\to\infty}|x_{j,n}|-t_n-t_{j,n}<\overline{A}
\end{equation} 
and we see, fixing $\eps$ small enough, that for any $s$, $Q_{n,\eps}(s)$ is empty for large $n$, yielding \eqref{E78}.
\item If $\ell<0$, then since $t_{j,n}\geq 0$ for large $n$, we obtain again \eqref{F77'} and thus \eqref{E78}.
\item 
In the case $\ell=0$ and $\lim_{n}|t_{j,n}|>0$, we obtain again \eqref{F77'} and \eqref{E78}. Finally, we assume $\ell=0$ and 
$$\lim_{n\to\infty}|t_{j,n}|=0.$$
Letting, after extraction $\omega_{\infty}=\lim_{n}\frac{x_{j,n}}{|x_{j,n}|}$, we see, fixing $y$, and using that $|x_{j,n}|$ goes to infinity, that
\begin{equation}
 \label{F78}
 \lim_{n\to\infty} \left|\widehat{\Big(\lambda_{j,n}y+x_{j,n},e_1\Big)}\right|=\left|\widehat{(\omega_{\infty},e_1)}\right|.
\end{equation} 
Using that $\lim_n t_{j,n}=\lim_n \lambda_{j,n}=0$, we must have 
$$\lim_{n\to\infty}\left|x_{j,n}-(\overline{A}+t_n)e_1\right|>0$$
and thus (using that $\ell=0$), $\omega_{\infty}\neq e_1$. By \eqref{F78}, we see that \eqref{E78} holds again.
\end{itemize}
\EMPH{Case 3}
\begin{equation}
 \label{E79}
 \lim_{n\to\infty}\lambda_{j,n}=\lambda_{\infty}\in (0,+\infty).
\end{equation} 
In this case we cannot prove \eqref{E78} for a fixed $\eps$. We prove \eqref{E75} by contradiction, assuming that for all $\eps>0$,
\begin{equation}
 \label{E79'}
 \limsup_{n\to\infty} |I_{n,\eps}|\geq \frac{\delta_2}{4J_1}.
\end{equation} 
Then we can find a sequence of positive numbers $\{\eps_k\}_k\to 0$, a sequence of integers $\{n_k\}_k\to+\infty$ such that 
\begin{equation}
 \label{E80} 
 \forall k,\quad I_{n_k,\eps_k}\geq \frac{\delta_2}{8J_1}.
\end{equation} 
As a consequence, we see that $Q_{n_k,\eps_k}(s)$ is included in an annulus of length $\lesssim \eps_k$ (using \eqref{E79}) and thus $\indic_{Q_{n_k,\eps_k}(s)}(y)$ goes to $0$ for almost every $y$, which contradicts \eqref{E79'} and concludes the proof.
 \end{proof}
\section{Concentration in a direction and virial type identity}
\label{S:virial}
In this part we prove the following:
\begin{prop}
\label{P:E12}
 Let $u$ be as in Theorem \ref{T:E1}, and assume that the conclusion of Theorem \ref{T:E1} does not hold. Let $\overline{A}$ be given by Proposition \ref{P:E5}. Then
there exists a sequence $\{t_n'\}_n\to+\infty$ such that, for some $\alpha>0$,
\begin{equation}
 \label{E81}
 \limsup_{n\to\infty} \int_{|x-(t_n'+\overline{A}e_1)|<\alpha} \left|\partial_{x_1}u(t_n')+\partial_tu(t_n')\right|^2+\sum_{j=2}^J |\partial_{x_j}u(t_n')|^2\,dx<\eps.
\end{equation} 
\end{prop}
\begin{remark}
We will use Proposition \ref{P:E12} to prove that $e_1\notin \SSS$.
 Of course one could write an analogue of Proposition \ref{P:E12} adapted to another direction than $e_1$. However, since we will use the spherical symmetry of equation \eqref{NLW} to reduce to the direction $e_1$, Proposition \ref{P:E12} will be sufficient for our purpose.
\end{remark}
We start with a few lemmas.
\begin{lemma}
 \label{L:E13}
 Let $\{\tau_n\}_n\to+\infty$ be any sequence. Denote by
 \begin{equation}
  \label{Edefrho}
\rho(t,x)=
|u(t,(t+\overline{A})e_1+x)|^{\frac{2N}{N-2}}+|\nabla_{t,x}u(t,(t+\overline{A})e_1+x)|^2+\frac{1}{|x|^2}|u(t,(t+\overline{A})e_1+x)|^2.
 \end{equation} 
Then, after extraction of a subsequence, there exists a non-negative Radon measure $\tilde{\mu}$ on $\RR^N$ such that 
 \begin{gather}
  \label{E82}
 \rho(\tau_n,\cdot)\xrightharpoonup[n\to\infty]{} \tilde{\mu}\\
 \label{E83}
 \supp \tmu \subset\{x_1\leq 0\}
 \end{gather} 
 \end{lemma}

 \begin{proof}
 Since the sequence
 $ \left\{\rho(\tau_n,\cdot)\right\}_n$
 is bounded in $L^1(\RR^N)$, one can always extract a subsequence so that \eqref{E82} holds for some positive finite Radon measure $\tmu$ on $\RR^N$. We just need to prove \eqref{E83}.
 Let $\eps>0$. By Proposition \ref{P:E5},
 \begin{equation}
  \label{E85}
  \lim_{n\to\infty} \int_{x_1\geq \eps} \left|\nabla_{t,x}u(t_n,x+(t_n+\overline{A})e_1)-\nabla_{t,x}v_{\lin}(t_n,x+(t_n+\overline{A})e_1)\right|^2\,dx=0,
 \end{equation} 
and
\begin{equation}
 \label{E86}
  \lim_{n\to\infty} \int_{x_1\geq \eps}  \frac{1}{|x|^2}\left|u(t_n,x+(t_n+\overline{A})e_1)\right|^2+\left|u(t_n,x+(t_n+\overline{A})e_1)\right|^{\frac{2N}{N-2}}\,dx=0.
 \end{equation} 
 We are thus reduced to proving that for all $\varphi\in C_0^{\infty}(\RR^N)$,
 \begin{equation}
  \label{E87}
  \lim_{n\to\infty} \int |\nabla_{t,x}v_{\lin}(t_n,x+(t_n+\overline{A})e_1)|^2\varphi(x)\,dx=0.
 \end{equation} 
 Consider the radiation field $G\in L^2(\RR\times S^{N-1})$ associated to $v_{\lin}$ (see Theorem \ref{T:B2}). We have 
 \begin{multline*}
\int \left|\nabla_{t,x}v_{\lin}(t_n,x+(t_n+\overline{A})e_1)\right|^2\varphi(x)\,dx=
\int \left|\nabla_{t,x}v_{\lin}(t_n,x)|^2\varphi(x-(t_n+\overline{A})e_1)\right|\,dx\\
=2\int_0^{+\infty} \int_{S^{N-1}} |G(r-t_n,\omega)|^2\varphi(r\omega-(t_n+\overline{A})e_1)\,d\omega\,dr+o(1)\\
=2\int_{-t_n}^{+\infty}\int_{S^{N-1}} |G(\eta,\omega)|^2\varphi\left( (t_n+\eta)\omega-(t_n+\overline{A})e_1 \right)\,d\omega\,d\eta+o(1).
 \end{multline*}
 Next, notice that $\varphi\left((t_n+\eta)\omega-(t_n+\overline{A})e_1\right)$ goes to $0$ for all $(\eta,\omega)\in \RR\times \left(S^{N-1}\setminus \{e_1\}\right)$. This proves \eqref{E87}, and thus, in view of \eqref{E85} and \eqref{E86} and since $\eps$ can be taken arbitrarily small, \eqref{E83}. 
\end{proof}
In the sequel, we will decompose $\tmu$ as 
\begin{equation}
 \label{E89}
 \tmu=c_0\delta_{\{x=0\}}+\mu,
\end{equation} 
where $\delta_{\{x=0\}}$ is the Dirac measure at $x=0$, $c_0=\tmu(\{0\})\geq 0$, and $\mu$ is a non-negative Radon measure such that $\mu(\{0\})=0$. If $e_1$ is a singular direction, we can prove, using Lemma \ref{L:E11}, that $c_0>0$ but this will not be used in the sequel.
\begin{lemma}
 \label{L:E14}
 Let $\{t_n\}_n$ be a sequence of times going to $+\infty$ as $n$ goes to infinity, and $\eps>0$. Then (after extraction of a subsequence from $\{t_n\}_n$), there exists $\alpha>0$ and two non-negative Radon measures $\mu_0$ and $\mu_1$ on $\RR^N$, and non-negative real numbers $c_0$ and $c_1$ such that 
 \begin{gather}
  \label{E89'}
  \mu_0(\{0\})=\mu_1(\{0\})=0\\
  \label{E90}
 \rho(t_n,\cdot)\xrightharpoonup[n\to\infty]{} \mu_0+c_0\delta_0\\
 \label{E91}
 \rho(t_n-\alpha/10,\cdot)\xrightharpoonup[n\to\infty]{} \mu_1+c_1\delta_1\\
 \label{E92}
 \supp \mu_{j}\subset\{x_1\leq 0\},\quad j=0,1\\
 \label{E93}
 \mu_0\left(\{|x|\leq \alpha\}\right)<\eps\\
 \label{E94}
 \mu_1\left(\left\{|x|\leq \frac{7\alpha}{10}\right\}\right)<\eps.
 \end{gather}
\end{lemma}
\begin{proof}
 By Lemma \ref{L:E13}, there exist a subsequence of $\{t_n\}_n$ and a non-negative measure $\mu_0$ that satisfies \eqref{E89'}, \eqref{E90} and \eqref{E92}. Since $\mu_0$ outer regular, we have:
 $$ 0=\mu_0(\{0\})=\inf_{R>0} \mu_0\left(B(0,R)\right),$$
 and we can find $\alpha>0$ such that 
 \begin{equation}
  \label{E95}
  \mu_0(\{|x|<\alpha\})<\eps/C_0
 \end{equation} 
 for some large constant $C_0>0$ to be specified. By Lemma \ref{L:E13} again, with $\tau_n=t_n-\alpha/10$, there exists (extracting subsequences) a measure $\mu_1$ that satisfies \eqref{E89'}, \eqref{E91} and \eqref{E92}. It remains to check that \eqref{E94} holds. 
 
 Let $\delta>0$. Since $\mu_0\left(\left\{|x|<\alpha\right\}\right)<\eps/C_0$, we have
 \begin{equation}
  \label{E96}
  \limsup_{n\to\infty} \int_{\delta\leq |x|\leq 19\alpha/20} \rho(t_n,x)\,dx<\eps/C_0.
 \end{equation} 
 As a consequence, if $\varphi$ is a $C^{\infty}$ function equal to $1$ for $|x|\geq 2$ and to $0$ for $|x|\leq 1$, and $\psi$ is a $C^{\infty}$ function equal to $1$ for $|x|\leq 9$ and $0$ for $|x|\geq 9.5$, we have
 \begin{equation*}
  \limsup_{n\to\infty} \int \left|\nabla_{t,x}\left(\varphi\left( \frac{x}{\delta} \right)\psi\left( \frac{x}{\alpha} \right)u\left(t_n,(\overline{A}+t_n)e_1+x\right)\right)\right|^2\,dx<C\eps/C_0.
 \end{equation*} 
 Using finite speed of propagation and small data theory, we deduce (choosing $C_0$ large enough):
 \begin{equation}
  \label{E97}
  \limsup_{n\to\infty} \sup_{t_n-\alpha/10\leq t\leq t_n} \int_{\frac{\alpha}{10}+3\delta\leq |x|\leq \frac{4\alpha}{5}}\rho(t,t_n,x)\,dx<\eps, 
 \end{equation} 
 where 
 \begin{equation}
 \label{Edefrhobis}
 \rho(t,t_n,x)=|u(t,(t_n+\overline{A})e_1+x)|^{\frac{2N}{N-2}}+|\nabla_{t,x}u(t,(t_n+\overline{A})e_1+x)|^2+\frac{1}{|x|^2}|u(t,(t_n+\overline{A})e_1+x)|^2. 
 \end{equation} 
 This proves, since $\rho\left(t_n-\alpha/10,t_n,x-\frac{\alpha}{10}e_1\right)=\rho\left(t_n-\frac{\alpha}{10},x\right)$,
 \begin{equation}
  \label{E98}
  \mu_1\left(\left\{\frac{\alpha}{10}+3\delta\leq \left|x-\frac{\alpha}{10}e_1\right|\leq \frac{4\alpha}{5}\right\}\right)<\eps.
 \end{equation} 
 
\begin{figure}
\caption{}
\label{Fig3}
\includegraphics{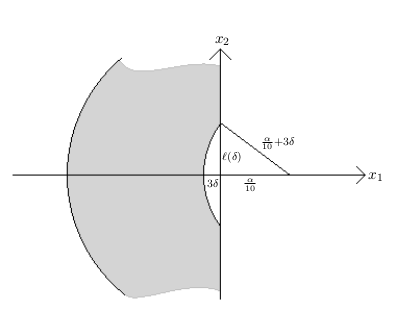} 
\end{figure} 
 The distance $\ell(\delta)$ in Figure \ref{Fig3} is equal to 
 $$\ell(\delta)=\sqrt{\left( \frac{\alpha}{10}+3\delta \right)^2-\left(\frac{\alpha}{10}\right)^2}\underset{\delta\to 0}{\longrightarrow} 0.$$
 Hence, from the figure
 $$\left\{x\;:\; 0<|x|<\frac{7\alpha}{10}\text{ and }x_1\leq 0\right\}\subset \bigcup_{\delta>0} \left\{3\delta +\frac{\alpha}{10}<\left|x-\frac{\alpha}{10}e_1\right|<\frac{4\alpha}{5}\right\}.$$
 Thus by \eqref{E98}, \eqref{E89'} and \eqref{E92},
 \begin{equation}
  \label{E99}
  \mu_1\left(\left\{ |x|\leq \frac{7\alpha}{10}\right\}\right)<\eps.
 \end{equation} 
\end{proof}
\begin{remark}
 \label{R:E14'}
 Let 
 \begin{equation}
  \label{E100}
 R_n(\eps):=\sup_{t_n-\alpha/10\leq t\leq t_n} \int_{\frac{\alpha}{5}\leq |x|\leq \frac{4\alpha}{5}} \rho(t,t_n,x)\,dx,
 \end{equation} 
 where $\rho(t,t_n,x)$ is defined in \eqref{Edefrhobis}. It follows from the proof of Lemma \ref{L:E14} that for large $n$
 \begin{equation}
  \label{E101}
  R_n(\eps)<2\eps.
 \end{equation} 
 (See inequality \eqref{E97} with $\delta=\alpha/30$).
\end{remark}
\begin{proof}[Proof of Proposition \ref{P:E12}]
 Let $\{t_n\}_n\to+\infty$. Let $\alpha$, $\mu_0$, $\mu_1$ be given by Lemma \ref{L:E14}, corresponding to $\eps$. 
 
 Let $\varphi\in C_0^{\infty}(\RR^N)$ such that 
 \begin{equation}
  \label{E102}
  \varphi(x)=0\text{ if }|x|\geq \frac 12,\quad \varphi(x)=1\text{ if } |x|\leq \frac 14.
 \end{equation} 
 Let 
 \begin{align*}
  u_n(t,x)&=u(t,x+(t_n+\overline{A})e_1),\quad \varphi_{\alpha}(x)=\varphi(x/\alpha)\\
  e(u_n)(t,x)&=\frac{1}{2}|\nabla u_n|^2+\frac 12(\partial_t u_n)^2-\frac{N-2}{2N} |u_n|^{\frac{2N}{N-2}}\\
  a_n(t)&=\int (\partial_tu_n)^2\varphi_{\alpha} ,\quad b_n(t)=\int|\nabla u_n|^2\varphi_{\alpha}\\
  c_n(t)&=\int |u_n|^{\frac{2N}{N-2}}\varphi_{\alpha},\quad d_n(t)=\int \partial_{x_1}u_n\partial_t u_n\varphi_{\alpha}.
 \end{align*}
Observe that 
\begin{align}
\label{lim_un1}
 \left|\nabla_{t,x}u_n(t_n,x)\right|^2+\frac{1}{|x|^2} |u_n(t_n,x)|^2+|u_n(t_n,x)|^{\frac{2N}{N-2}}&\xrightharpoonup[n\to\infty]{}c_0\delta_0+\mu_0\\
\label{lim_un2}
 \left|\nabla_{t,x}u_n\left(t_n-\frac{\alpha}{10},x-\frac{\alpha}{10}e_1\right)\right|^2+\frac{1}{|x|^2} \left|u_n\left(t_n-\frac{\alpha}{10},x-\frac{\alpha}{10}e_1\right)\right|^2 &\\
 \notag
 +\left|u_n\left(t_n-\frac{\alpha}{10},x-\frac{\alpha}{10}e_1\right)\right|^{\frac{2N}{N-2}}&\xrightharpoonup[n\to\infty]{}c_1\delta_1+\mu_1.
 \end{align}
We denote the average values of $a_n$, $b_n$, $c_n$ and $d_n$ between $t_n-\alpha/10$ and $t_n$ by $\overline{a}_n$, $\overline{b}_n$, $\overline{c}_n$ and $\overline{d}_n$ respectively: 
$$ \overline{a}_n:=\frac{10}{\alpha}\int_{t_n-\alpha/10}^{t_n} a_n(t)\,dt,$$
and similarly for $\overline{b}_n$, $\overline{c}_n$ and $\overline{d}_n$.

\EMPH{Step 1} By explicit computations, using \eqref{E100}, we obtain, for $t\in [t_n-\alpha/10,t_n]$,
\begin{align}
 \label{E103}
 \frac{d}{dt} \int u_n\partial_tu_n\varphi_{\alpha}&=a_n(t)-b_n(t)+c_n(t)+\OOO(\eps)\\
 \label{E104}
 \frac{d}{dt} \int x\cdot \nabla u_n\partial_t u_n&=-\frac{N}{2} a_n(t)+\left( \frac{N}{2}-1 \right)b_n(t)-\left( \frac{N}{2}-1 \right)c_n(t)+\OOO(\eps)\\
 \label{E105}
 \frac{d}{dt}\int x_1 e(u_n)\varphi_{\alpha}&=-d_n+\OOO(\eps)\\
 \label{E106}
 \frac{d}{dt} \int \partial_{x_1}u_n\partial_tu_n\varphi_{\alpha}&=d_n'(t)=\OOO(\eps)\\
 \label{E107}
 \frac{d}{dt}\int e(u_n)\varphi_{\alpha}&=\frac{1}{2}a_n'(t)+\frac{1}{2}b_n'(t)-\frac{N-2}{2N}c_n'(t)=\OOO(\eps),
\end{align}
where $\OOO(\eps)$ is uniform with respect to $t\in [t_n-\alpha/10,t_n]$. 

These computations are classical. The only thing to check is the bound on the remainder. For example, we have (using Einstein's summation convention)
\begin{multline*}
\frac{d}{dt}\int x_k\partial_{x_k}u_n\partial_tu_n\varphi_{\alpha}=-\frac{N}{2}a_n+\left( \frac{N}{2}-1 \right)b_n -\left( \frac{N}{2}-1 \right)c_n+\\
\frac 12\int x_k\partial_{x_k}\varphi_{\alpha}(\partial_{x_j}u_n)^2
-\frac 12\int x_k\partial_{x_k}\varphi_{\alpha}(\partial_tu_n)^2-\int \partial_{x_k} \varphi_{\alpha} x_j\partial_{x_j} u_n\partial_{x_k}u_n-\frac{N-2}{2N} \int x_j\partial_{x_j}\varphi_{\alpha} |u|^{\frac{2N}{N-2}}
\end{multline*}
and \eqref{E104} follows from \eqref{E101} and the bound $|x_j\partial_{x_k}\varphi_{\alpha}|=\left|\frac{x_j}{\alpha}\partial_{x_k}\varphi\left( \frac{x}{\alpha} \right)\right|\lesssim 1$.

\EMPH{Step 2. Approximate conservation laws}
We prove that for large $n$:
\begin{gather}
\label{E109}
\sup_{t_n-\alpha/10\leq t\leq t_n} \left|d_n(t)-\overline{d}_n\right|\lesssim \eps\,\alpha\\
\label{E110}
\sup_{t_n-\alpha/10\leq t\leq t_n} \left|\frac{1}{2} a_n(t)+\frac{1}{2}b_n(t)-\frac{N-2}{2N}c_n(t)-\frac{1}{2}\overline{a}_n-\frac{1}{2}\overline{b}_n+\frac{N-2}{2N}\overline{c}_n\right|\lesssim \eps\,\alpha.
\end{gather}
Indeed, by \eqref{E106}, there exists a constant $C>0$ such that for all $\tau_1,\tau_2$ with $t_n-\alpha/10\leq \tau_1,\tau_2\leq t_n$,
\begin{equation}
 \label{E111}
 \left|d_n(\tau_1)-d_n(\tau_2)\right|\leq C\eps\alpha,
\end{equation} 
and \eqref{E109} follows. The proof of \eqref{E110} is similar, using \eqref{E107} instead of \eqref{E106}.

\EMPH{Step 3} We prove that for large $n$,
\begin{gather}
 \label{E112}
 -\overline{d}_n=\frac{1}{2}\overline{a}_n+\frac{1}{2}\overline{b}_n-\frac{N-2}{2N}\overline{c}_n+\OOO(\eps)\\
 \label{E113}
 \overline{a}_n-\overline{b}_n+\overline{c}_n=\OOO(\eps)\\
 \label{E114}
 -\frac{N}{2}\overline{a}_n+\left( \frac{N}{2}-1 \right)(\overline{b}_n-\overline{c}_n)=\overline{d}_n+\OOO(\eps).
\end{gather}
\EMPH{Proof of \eqref{E112}} Integrate \eqref{E105} between $t_n-\alpha/10$ and $t_n$ to obtain:
\begin{equation}
 \label{E115}
 -\overline{d}_n=\frac{10}{\alpha}\int \varphi_{\alpha}(x)x_1e(u_n)(t_n,x)\,dx-\frac{10}{\alpha}\int \varphi_{\alpha}(x)x_1e(u_n)\left( t_n-\frac{\alpha}{10},x \right)\,dx+\OOO(\eps).
\end{equation} 
We have for large $n$,
\begin{equation}
 \label{E116}
 \int \varphi(x)x_1e(u_n)(t_n,x)\,dx=\OOO(\alpha\eps).
\end{equation} 
Indeed, by the definition of $\mu_0$ (see \eqref{lim_un1}),
\begin{equation*}
 \lim_{n\to\infty} \int \varphi_{\alpha}(x)|x_1|\left(|\nabla_{t,x}u_n|(t_n,x)+\frac{1}{|x|^2} |u(t_n,x)|^2+|u(t_n,x)|^{\frac{2N}{N-2}}\right)\,dx=\int |x_1|\varphi_{\alpha}(x) d\mu_0(x),
\end{equation*} 
and $\left|\int |x_1|\varphi_{\alpha}(x) d\mu_0(x)\right|\lesssim \eps\alpha$ by Lemma \ref{L:E14} and the bound $\left||x_1|\varphi_{\alpha}(x)\right|\lesssim \alpha$. 

Furthermore, by the change of variable $x=y-\frac{\alpha}{10}e_1$, 
\begin{multline*}
 \int \varphi_{\alpha}(x)x_1e(u_n)\left( t_n-\frac{\alpha}{10},x \right)\,dx=\int \varphi_{\alpha}\left(y-\frac{\alpha}{10}e_1\right)\left(y_1-\frac{\alpha}{10}\right)e(u_n)\left( t_n-\frac{\alpha}{10},y-\frac{\alpha}{10}e_1 \right)\,dy\\
 =\int y_1\varphi_{\alpha}\left(y-\frac{\alpha}{10}e_1\right)e(u_n)\left( t_n-\frac{\alpha}{10},y
 -\frac{\alpha}{10}e_1 \right)\,dy\\
 -\frac{\alpha}{10} \int \varphi_{\alpha}\left(y-\frac{\alpha}{10}e_1\right)e(u_n)\left( t_n-\frac{\alpha}{10},y-\frac{\alpha}{10}e_1 \right)\,dy.
\end{multline*}
By \eqref{lim_un2} and Lemma \ref{L:E14},
the limit of the first term is bounded (up to a multiplicative constant) by 
$$\left|\int |y_1|\varphi_{\alpha}(y-\alpha/10 e_1) d\mu_1(y)\right|\lesssim \eps\alpha.$$
The second term can be rewritten
\begin{multline*}
 -\frac{\alpha}{10} \int \varphi_{\alpha}\left(y-\frac{\alpha}{10}e_1\right)e(u_n)\left( t_n-\frac{\alpha}{10},y-\frac{\alpha}{10}e_1 \right)\,dy=-\frac{\alpha}{10} \int \varphi_{\alpha}\left(x\right)e(u_n)\left( t_n-\frac{\alpha}{10},x\right)\,dx\\
 =-\frac{\alpha}{10} \left(\frac{1}{2} a_n\left( t_n-\frac{\alpha}{10} \right)+\frac{1}{2}b_n\left(t_n-\frac{\alpha}{10}\right) -\frac{N-2}{2N}c_n\left( t_n-\frac{\alpha}{10} \right)\right).
\end{multline*}
Going back to \eqref{E115}, we obtain
\begin{equation*}
 -\overline{d}_n=\frac{1}{2}a_n\left( t_n-\frac{\alpha}{10} \right)+\frac{1}{2}b_n\left( t_n-\frac{\alpha}{10} \right)-\frac{N-2}{2N}c_n\left( t_n-\frac{\alpha}{10} \right)+\OOO(\eps),
\end{equation*} 
which yields \eqref{E112} in view of the approximate conservation of the energy \eqref{E110} proved in Step 2.

\EMPH{Proof of \eqref{E113}} Integrating \eqref{E103} between $t_n-\alpha/10$ and $t_n$, we are reduced to prove the bound:
\begin{equation}
 \label{E117}
 \left|\int \varphi_{\alpha}(x) u_n(t_n,x)\partial_tu_n(t_n,x)\,dx-\int \varphi_{\alpha}(x) u_n\left( t_n-\frac{\alpha}{10},x \right)\partial_tu_n\left( t_n-\frac{\alpha}{10},x \right)\,dx\right|\lesssim \eps\alpha
\end{equation}
for large $n$.
This follows easily from Lemma \ref{L:E14}. For example
\begin{multline*}
 \left|\int \varphi_{\alpha}(x) u_n\left( t_n-\frac{\alpha}{10},x \right)\partial_tu_n\left( t_n-\frac{\alpha}{10},x \right)\,dx\right|\\
 \leq \left(\int \varphi_{\alpha}(x) \frac{1}{\left|x+\frac{\alpha}{10}e_1\right|}|u_n|^2\left(t_n-\frac{\alpha}{10},x\right)\,dx   \right)^{\frac{1}{2}}\\ \times\left(\int \varphi_{\alpha}(x) \left|x+\frac{\alpha}{10}e_1\right| |\partial_t u_n|^2\left(t_n-\frac{\alpha}{10},x\right)\,dx   \right)^{\frac{1}{2}}
\end{multline*}
and 
\begin{multline*}
 \int \varphi_{\alpha}(x) \frac{1}{\left|x+\frac{\alpha}{10}e_1\right|}|u_n|^2\left(t_n-\frac{\alpha}{10},x\right)\,dx\\
 =\int \varphi_{\alpha}\left(y-\frac{\alpha}{10}e_1\right) \left|y\right| \frac{1}{\left|y\right|^2}|u_n|^2\left(t_n-\frac{\alpha}{10},y-\frac{\alpha}{10}e_1\right)\,dy\\
 \lesssim \left|\int |y|\varphi_{\alpha}\left( y-\frac{\alpha}{10}e_1 \right)\,d\mu_1(y)\right|\lesssim \alpha \eps
\end{multline*}
by \eqref{lim_un2} and Lemma \ref{L:E14}. The estimates of the other terms are similar. 

\EMPH{Proof of \eqref{E114}} We integrate \eqref{E104} between $t_n-\alpha/10$ and $t_n$, obtaining
\begin{multline}
 \label{E118}
 -\frac{N}{2} \overline{a}_n+\left( \frac{N}{2}-1 \right)\left(\overline{b}_n-\overline{c}_n \right)\\
 =\OOO(\eps)+\frac{10}{\alpha}\left( -\int \varphi_{\alpha} x\cdot \nabla u_n\left( t_n-\frac{\alpha}{10} \right)\partial_t u_n\left( t_n-\frac{\alpha}{10} \right)+\int \varphi_{\alpha} x\cdot\nabla u_n(t_n)\partial_tu_n(t_n) \right).
\end{multline} 
By computations that are similar to the ones above, the right-hand side of \eqref{E118} is given by 
\begin{equation*}
 \OOO(\eps)+\int \varphi_{\alpha}(x)\partial_{x_1}u_n\left(t_n-\frac{\alpha}{10},x\right)\partial_tu_n\left(t_n-\frac{\alpha}{10},x\right)
 =\OOO(\eps) +d_n\left( t_n-\frac{\alpha}{10} \right)=\OOO(\eps)+\overline{d}_n.
\end{equation*}
At the last line we have used Step 2 to replace $d_n(t_n-\alpha/10)$ by $\overline{d}_n+\OOO(\eps)$.

\EMPH{Step 4. End of the proof}
Subtracting \eqref{E112} and \eqref{E114}, we obtain
$$\left( \frac{N}{2}-\frac{1}{2} \right)\left( \overline{b}_n-\overline{a}_n \right)-\left( \frac{N}{2}-\frac{1}{2}-\frac{1}{N} \right)\overline{c}_n=\OOO(\eps)$$
for large $n$.
Adding $\left( \frac{N}{2}-\frac{1}{2} \right)$\eqref{E113} we deduce
\begin{equation}
 \label{E119}
 \overline{c}_n=\OOO(\eps).
\end{equation} 
Combining this with \eqref{E112}, we obtain
\begin{equation}
 \label{E121}
 \frac{1}{2}\overline{a}_n+\frac{1}{2}\overline{b}_n+\overline{d}_n=\OOO(\eps).
\end{equation} 
Hence, for large $n$,
\begin{equation}
 \label{E122}
 \frac{1}{\alpha}\int_{t_n-\alpha/10}^{t_n} \int \left( (\partial_tu_n(t,x)+\partial_{x_1}u_n(t,x))^2+|\nabla_{x'}u_n(t,x)|^2 \right)\varphi_{\alpha}(x)\,dx\,dt=\OOO(\eps),
\end{equation} 
where $\nabla_{x'}=(\partial_{x_2},\ldots,\partial_{x_N})$.

As a consequence, we obtain a sequence $\{t_n'\}\to\infty$ such that for all $n$, $t_n-\frac{\alpha}{10}\leq t_n'\leq t_n$ and
\begin{equation}
 \label{E123}
 \forall n,\quad \int \left( (\partial_tu+\partial_{x_1}u )^2(t_n',x)+|\nabla_{x'}u|^2(t_n',x)\right)\varphi\left( \frac{x-(\overline{A}+t_n)e_1}{\alpha} \right)\,dx \leq C\eps.
\end{equation} 
Since 
$$|x-(t_n'+A)e_1|\leq \frac{\alpha}{10}\Longrightarrow \left|x-(A+t_n)e_1\right|\leq \frac{\alpha}{5}\Longrightarrow \varphi\left( \frac{x-(A+t_n)e_1}{\alpha} \right)=1,$$
we obtain \eqref{E81} (renormalizing $\alpha$ and $\eps$). The proof of Proposition \ref{P:E12} is complete.
\end{proof}

\section{Elimination of singular points and end of the proof}
\label{S:end_of_proof}
We are now ready to conclude the proof of Theorem \ref{T:E1}. We will prove that $\SSS$ is empty, contradicting Proposition \ref{P:E7}. 

We argue by contradiction assuming (after a rotation in the space variable) $e_1\in \SSS$. We let, for $(f,g)\in (\hdot\times L^2)(\RR^N)$, 
\begin{equation}
 \label{E124}
 \|(f,g)\|^2_{e_1}=\|g+\partial_{x_1}f\|^2_{L^2}+\sum_{j=2}^{N}\|\partial_{x_j}f\|_{L^2}^2.
\end{equation} 
\begin{remark}
 \label{R:E14''}
 If $u_{\lin}$ is a solution to the linear wave equation with initial data in $\hdot\times L^2$, then $\left\|(u_{\lin}(t),\partial_tu_{\lin}(t)\right\|_{e_1}$ is independent of $t$. Indeed,
 \begin{equation}
  \label{E127}
  \left\|(u_{\lin}(t),\partial_tu_{\lin}(t))\right\|_{e_1}^2=\left\|\partial_{x_1}u_{\lin}(t)\right\|^2_{L^2}+\|\partial_tu_{\lin}(t)\|^2_{L^2}+ 2\int_{\RR^N} \partial_{x_1}u_{\lin}(t)\partial_tu_{\lin}(t)+\sum_{k=2}^N \|\partial_{x_k}u_{\lin}(t)\|^2_{L^2},
 \end{equation} 
 and the conservation of $\|(u_{\lin}(t),\partial_tu_{\lin}(t)\|_{e_1}$ follows from energy and momentum conservations.
\end{remark}
\begin{remark}
 \label{R:E14'''}
 Denote by $(\cdot,\cdot)_{e_1}$ the scalar product associated to the norm $\|\cdot\|_{e_1}$. 
 Let $\{(v_0^n,v_1^n)\}_n$ be a sequence bounded in $\hdot\times L^2$ that has a profile decomposition \profiles. Then
 $$j\neq k\Longrightarrow \lim_{n\to\infty}\left(\vec{U}_{\lin,n}^k(0),\vec{U}_{\lin,n}^j(0)\right)_{e_1}=0,\quad 1\leq j\leq J\Longrightarrow
 \lim_{n\to\infty}\left(\vec{U}_{\lin,n}^j(0),\vec{w}_n^J(0)\right)_{e_1}=0,$$
 where $w_n^J$ is the remainder of the profile decomposition (see \eqref{wnJ}). This follows from Remark \ref{R:E14''} and the same argument than the one used to prove the orthogonality of the energy of the profiles (see Lemma 2.3 of \cite{DuJiKeMe16P} for a proof).
 
 As a consequence, the following Pythagorean expansion holds: for all $J\geq 1$, 
 \begin{equation}
  \label{E128}
  \left\|\left(v_0^n,v_1^n\right)\right\|^2_{e_1}=\sum_{j=1}^J \left\|U_{\lin}^j(0)\right\|^2_{e_1}+\left\|w_n^J(0)\right\|^2_{e_1}+o(1)
 \end{equation} 
 as $n\to\infty$. 
\end{remark}
We will use the following claim, proved in the appendix:
\begin{claim}
 \label{Cl:E15}
 Let $\beta>0$, $M>0$. Then there exists $\eps=\eps(M,\beta)>0$ such that if $(v_0,v_1)\in \hdot\times L^2$ satisfies 
 \begin{equation}
  \label{E125}
  \left\|(v_0,v_1)\right\|_{e_1}\leq \eps\text{ and } \|(v_0,v_1)\|_{\hdot\times L^2}\leq M,
 \end{equation} 
 then
 \begin{equation}
  \label{E126}
  \left\|S_L(t)(v_0,v_1)\right\|_{S(\RR)}\leq \beta.
 \end{equation} 
\end{claim}
Let $\eps>0$ given by Claim \ref{Cl:E15} with $\beta=\delta_3/2$, $\delta_3$ given by Lemma \ref{L:E11} and 
$$M=\sup_{t\geq 0}\left\|\vec{u}(t)\right\|_{\hdot\times L^2}.$$
By Proposition \ref{P:E12}, there exists a sequence $\{t_n'\}\to +\infty$ and $\alpha>0$ such that 
\begin{equation}
 \label{E140}
 \limsup_{n\to+\infty} \int_{\left|x-(t_n'+\overline{A})e_1\right|<\alpha}\left( \partial_{x_1}u(t_n')+\partial_tu(t_n')\right)^2+\sum_{j=2}^J |\partial_{x_j}u(t_n')|^2\,dx\leq \eps.
\end{equation} 

By Lemma \ref{L:E11}, there exists a subsequence of $\{t_n'\}$, that we will still denote by $\{t_n'\}_n$, such that $\{u(t_n')\}_n$ has a profile decomposition \profiles with the following properties:
\begin{gather}
 \label{E141}
 \lim_{n\to\infty} x_{1,n}-(\overline{A}+t_n')e_1=0\\
 \label{E142}
 \lim_{n\to\infty} \lambda_{1,n}=\lim_{n\to\infty}t_{1,n}=0\\
 \label{E143}
 \left\|U_{\lin}^1\right\|_{S(\RR)}\geq \delta_3.
\end{gather}
By Remark \ref{R:E14'''},
\begin{equation}
 \label{E144}
 \lim_{n\to\infty}
 \left(\vec{u}(t'_n),\vU^1_{\lin,n}(0)\right)_{e_1}=\left\|\vU_{\lin}^1(0)\right\|^2_{e_1}.
\end{equation} 
Next, notice that it follows from \eqref{E141} and \eqref{E142} that
\begin{equation}
 \label{E145}
 \lim_{n\to \infty} \int_{|x-(t_n'+\overline{A})e_1|\geq \alpha}\left|\nabla_{t,x}U_{\lin,n}^1(0)\right|^2\,dx=0.
\end{equation} 
Indeed, this integral can be rewritten
\begin{equation}
 \label{E146}
 \int_{|\lambda_{1,n}y+o(1)|\geq \alpha} \left|\nabla_{t,x}U^1_{\lin,n}\left( \frac{-t_{1,n}}{\lambda_{1,n}},y \right)\right|^2\,dy,
\end{equation} 
where $o(1)=x_{1,n}-(\overline{A}+t_n')e_1$ goes to $0$ as $n$ goes to infinity by \eqref{E141}. The desired limit \eqref{E145} follows immediately if $-t_{1,n}/\lambda_{1,n}$ is bounded. If not, say if 
$\lim_{n\to \infty}-t_{1,n}/\lambda_{1,n}=-\infty$ after extraction of a subsequence, we can rewrite \eqref{E146} as 
$$ 2\int_{|\lambda_{1,n}r+o(1)|\geq \alpha}\int_{S^{N-1}} \left|G\left( r+\frac{t_{1,n}}{\lambda_{1,n}},\omega \right)\right|^2\,d\omega\,dr$$
where $G\in L^2(\RR\times S^{N-1})$ is the radiation field associated to $U^1_{\lin}$ as $t\to -\infty$ (see Theorem \ref{T:B2}). This integral goes to $0$ as $n$ goes to infinity since 
$$ \left|r+\frac{t_{1,n}}{\lambda_{1,n}}\right|\geq \frac{\alpha+o(1)}{\lambda_{1,n}}\underset{n\to\infty}{\longrightarrow} +\infty$$
on the domain of integration (using $\lim_nt_{1,n}=0$). This proves \eqref{E145}.

By \eqref{E145}, as $n\to\infty$,
\begin{multline*}
 \left(\vec{u}(t_n'),\vU_{\lin}^1(0)\right)_{e_1} \\
 \leq \left\| \vU_{\lin}^1(0)\right\|_{e_1}\sqrt{\int_{|x-(t_n'+\overline{A})e_1|\leq \alpha}\left|\partial_{x_1}u(t_n')+\partial_tu(t_n')\right|^2+\sum_{k=2}^N |\partial_{x_k} u(t_n')|^2}+o(1).
\end{multline*}
Combining with \eqref{E140}, we obtain
\begin{equation*}
 \left(\vec{u}(t_n'),\vU_{\lin,n}^1(0)\right)_{e_1}\leq \eps \left\|U_{\lin}^1(0)\right\|_{e_1}+o(1)\text{ as }n\to\infty.
\end{equation*} 
Hence by \eqref{E144},
\begin{equation}
 \label{E147}
\left\|\vU_{\lin}^1(0)\right\|_{e_1}\leq \eps.
\end{equation} 
But then by the definition of $\eps$ (from Claim \ref{Cl:E15} with $\beta=\frac{\delta_3}{2}$ and $M=\sup_{t\geq 0} \left\|\vec{u}(t)\right\|_{\hdot\times L^2}$):
$$ \left\|U_{\lin}^1\right\|_{S(\RR)}\leq \frac{\delta_3}{2},$$
contradicting \eqref{E143}. The proof is complete. \qed

\appendix
\section{Radiation field for linear wave equations}
\label{S:radiation}
In this appendix we prove Theorem \ref{T:B2}. 
\subsection{Introduction of a function space}
We start by reformulating this theorem in term of a space of functions on $\RR\times S^{N-1}$ that we will define now. Let
$$\dot{H}^1_{\eta}(\RR\times S^{N-1})=\left\{g\in C^0\left(\RR,L^2(S^{N-1})\right)\; :\; \int_{\RR\times S^{N-1}} |\partial_{\eta}g(\eta,\omega)|^2\,d\eta d\omega<\infty\right\}.$$
Let $g\in \dot{H}^1_{\eta}$. We note that $\left\|\partial_{\eta}g\right\|_{L^2(\RR\times S^{N-1})}=0$ if and only if there exists $a\in L^2(S^{N-1})$ such that for all $\eta\in \RR$, for almost all $\omega\in S^{N-1}$, $g(\eta,\omega)=a(\omega)$. We define $\dot{\HHH}^1_{\eta}$ as the quotient space of $\dot{H}^1_{\eta}$ by the equivalence relation:
$$ g\sim \tilde{g}\iff \exists a\in L^2(S^{N-1})\;:\; \forall \eta \in \RR, \, g(\eta,\omega)-\tilde{g}(\eta,\omega)=a(\omega)\quad \text{ for a.a. }\omega\in S^{N-1}.$$
We denote by $\overline{g}\in \dot{\HHH}^1_{\eta}$ the equivalence class of $g\in \dot{H}^1_{\eta}$, and we define the following norm on $\dot{\HHH}^1_{\eta}$:
$$ \left\|\overline{g}\right\|_{\dot{\HHH}^1_{\eta}}=\left\|\partial_{\eta}g\right\|_{L^2(\RR\times S^{N-1})}.$$
Then:
\begin{prop}
\label{P:1}
 The normed space $\dot{\HHH}^1_{\eta}$ is a Hilbert space, and $C^{\infty}_0\left(\RR\times S^{N-1}\right)$ is dense in $\dot{\HHH}^1_{\eta}$. The map $\overline{g}\mapsto \partial_{\eta}g$ is a bijective isometry from $\HHH^1_{\eta}$ to $L^2(\RR\times S^{N-1})$.
\end{prop}
We note that the proposition implies that $\dot{\HHH}^1_{\eta}$ is the closure of $C^{\infty}_0(\RR\times S^{N-1})$ for the norm $\|\partial_{\eta} \cdot\|_{L^2(\RR\times S^{N-1})}$. 
In view of Proposition \ref{P:1}, the following is equivalent to Theorem \ref{T:B2}:
\begin{theoint}
\label{T:B2'}
 Assume $N\geq 3$ and let $v$ be a solution of the linear wave equation \eqref{E3} with initial data $(v_0,v_1)\in \hdot\times L^2$. Then
 \begin{equation}
  \label{A4}
  \lim_{t\to+\infty} \left\|\frac{1}{r}\nabla_{\omega}v(t)\right\|_{L^2}+\left\|\frac 1r v(t)\right\|_{L^2}=0
 \end{equation} 
 and there exists a unique $\overline{g}\in \dot{\HHH}^1_{\eta}$ such that
 \begin{equation}
  \label{A1}
  \lim_{t\to+\infty} \int_0^{+\infty} \int_{S^{N-1}} \left|\partial_{r,t}\left( r^{\frac{N-1}{2}} v(t,r\omega)-g(r-t,\omega)\right)\right|^2\,d\omega\,dr=0.
 \end{equation} 
 Furthermore,
 \begin{equation} 
  \label{A2}
  E_{\lin}(v_0,v_1)=\int_{\RR\times S^{N-1}} |\partial_{\eta}g(\eta,\omega)|^2\,d\eta d\omega=\left\|\overline{g}\right\|^2_{\dot{\HHH}^1_{\eta}}
 \end{equation}
and the map
\begin{align*}
 (v_0,v_1)&\mapsto \sqrt{2}\overline{g}\\
\dot{H}^1\times L^2&\to \dot{\HHH}^1_{\eta}
 \end{align*}
is a bijective isometry.
\end{theoint}
In this Appendix \ref{S:radiation} we prove Theorem \ref{T:B2'} assuming Proposition \ref{P:1}. We postpone the proof of Proposition \ref{P:1} to Appendix \ref{S:funct}.
\subsection{The case of smooth, compactly supported functions}
\begin{lemma}
 \label{L:3} 
 Assume $(v_0,v_1)\in \left(C_0^{\infty}(\RR^N)\right)^2$ and let $v$ be the corresponding solution of \eqref{E3}. Then there exists $F\in C^{\infty}(\RR\times S^{N-1}\times [0,+\infty))$ such that 
 $$ \forall r>0,\; \forall \omega \in S^{N-1},\quad v(t,r\omega)=\frac{1}{r^{\frac{N-1}{2}}}F\left(r-t,\omega,\frac 1r\right).$$
\end{lemma}
\begin{proof}
 This is classical  (see \cite{Friedlander62}, \cite{LaPh89}, and also \cite{Alinhac09Bo} for this exact statement), and can be proved using the explicit form of the solution of \eqref{E3}, distinguishing between even and odd dimensions. We give a proof relying on the conformal transformation which is independent of the dimension.
 
 For $\rho\in \RR$, $\omega\in S^{N-1}$, $\sigma\in (0,+\infty)$ we let
 $$ F(\rho,\omega,\sigma)=\frac{1}{\sigma^{\frac{N-1}{2}}}v\left(\frac 1\sigma-\rho,\frac{\omega}{\sigma}\right).$$
 Since $v$ is smooth, $F$ is smooth on $\RR\times S^{N-1}\times (0,+\infty)$. We must prove that $F$ can be extended to a smooth function on $\RR\times S^{N-1}\times [0,+\infty)$.
 
 Let $t_0>0$. We claim that there exists a $C^{\infty}$ solution $w$ of the linear wave equation \eqref{E3}, depending on $t_0$ such that, for all $x,t$ such that $|x|>t-t_0$ and $t>t_0$,
 \begin{equation}
 \label{A5}
 v(t,x)=\frac{1}{\left(|x|^2-(t-t_0)^2\right)^{\frac{N-1}{2}}} w\left(\frac{t-t_0}{|x|^2-(t-t_0)^2},\frac{x}{|x|^2-(t-t_0)^2}\right).  
 \end{equation} 
Indeed, let $w$ be the solution of \eqref{E3} with $C^{\infty}$ initial data $(w_0,w_1)$ at $t=0$ given by
\begin{equation*}
 \left.
 \begin{aligned}
 w_0(y)&=\frac{1}{|y|^{N-1}} v\left(t_0,\frac{y}{|y|^2}\right)\\
w_1(y)&=\frac{1}{|y|^{N+1}} \partial_tv\left(t_0,\frac{y}{|y|^2}\right)
\end{aligned}
\right\} \text{ if }y\neq 0
 \end{equation*} 
 and $w_0(0)=w_1(0)=0$. Notice that since, by finite speed of propagation, $(v(t_0),\partial_t v(t_0))$ is compactly supported, the above definition yields $C^{\infty}$ functions on $\RR^N$ which are constant, equal to $0$, in a neighbourhood of the origin.

We note that
$$\tilde{w}(\tau,y)=\frac{1}{\left( |y|^2-\tau^2 \right)^{\frac{N-1}{2}}}v\left( t_0+\frac{\tau}{|y|^2-\tau^2},\frac{y}{|y|^2-\tau^2} \right)$$
defines, for $|y|>\tau$, a $C^{\infty}$ solution of the linear wave equation whose initial data at $\tau=0$ equals to $(w_0,w_1)$ (at least for $|y|\neq 0$). Hence, by finite speed of propagation,
$$|y|>\tau\Longrightarrow w(\tau,y)=\tilde{w}(\tau,y).$$
It remains to check that for $|x|>t-t_0$, $t>t_0$,
$$v(t,x)=\frac{1}{\left( |x|^2-|t-t_0|^2 \right)^{\frac{N-1}{2}}}\tilde{w}\left( \frac{t-t_0}{|x|^2-(t-t_0)^2},\frac{x}{|x|^2-(t-t_0)^2}\right).$$
This follows easily from the definition of $\tilde{w}$ and the change of variables
$$t=t_0+\frac{\tau}{|y|^2-\tau^2},\quad x=\frac{y}{|y|^2-\tau^2}.$$
As a consequence of \eqref{A5}, going back to the definition of $F$, we obtain
\begin{multline}
 \label{A6}
 \Big( \sigma>0,\;0<1-\sigma(\rho+t_0) \text{ and }\rho>-t_0 \Big)\Longrightarrow\\
 F(\rho,\omega,\sigma)=\frac{1}{\left( 2(\rho+t_0)-(\rho+t_0)^2\sigma \right)^{\frac{N-1}{2}}}w\left( \frac{1-\sigma(\rho+t_0)}{2(\rho+t_0)-(\rho+t_0)^2\sigma},\frac{\omega}{2(\rho+t_0)-(\rho+t_0)^2\sigma} \right)
\end{multline} 
However, the right-hand side of the second line of \eqref{A6} can be extended to a $C^{\infty}$ function in the open set
$$ \Big\{(\rho,\omega,\sigma)\in \RR\times S^{N-1}\times \RR\;:\; \sigma(\rho+t_0)<2 \text{ and }\rho>-t_0 \Big\}$$
which includes the set 
$(-t_0,+\infty)\times S^{N-1}\times \{0\}.$
As a consequence, $F$ can be extended to a $C^{\infty}$ function in a neighbourhood of the set 
$(1-t_0,+\infty)\times S^{N-1}\times [0,+\infty)$ and, since $t_0$ is arbitrarily large, to a neighbourhood of $\RR\times S^{N-1}\times [0,+\infty)$.
\end{proof}
\begin{lemma}
 \label{L:4}
 Let $N\geq 3$. Let $(v_0,v_1)\in \left(C_0^{\infty}(\RR^N)\right)^2$. Then
 \begin{equation}
  \label{A7}
  \lim_{t\to\infty}\left\|\frac{1}{r}\nabla_{\omega} v(t)\right\|_{L^2}+\left\|\frac{1}{r}v(t)\right\|_{L^2}=0
 \end{equation} 
 and there exists $g\in \dot{H}^1_{\eta}$ such that 
 \begin{equation}
  \label{A8}
  \lim_{t\to\infty}\left\|\partial_{r,t}\left( r^{\frac{N-1}{2}}v(t,r\omega)-g(r-t,\omega)\right)\right\|_{L^2((0,\infty)\times S^{N-1})}=0.
 \end{equation} 
\end{lemma}
Let us mention that Lemma \ref{L:4} follows quite easily from Lemma \ref{L:3} if $N\geq 3$ is odd (in this case, $g\in C_0^{\infty}(\RR\times S^{N-1})$). In the general case, we will need the following claim:
\begin{claim}
\label{Cl:5}
 Let $(v_0,v_1)\in \left( C_0^{\infty}(\RR^N) \right)^2$, $N\geq 3$ and $v$ the solution to the wave equation \eqref{E3}. Then $\|v(t)\|_{L^2}$ is bounded,
 \begin{equation}
 \label{Hardy}
 \lim_{t\to\infty} \int \frac{1}{|x|^2}|v(t,x)|^2\,dx=0
 \end{equation}
 and
 \begin{equation}
 \label{Morawetz}
 \lim_{R\to\infty}\limsup_{t\to\infty}\int_{|x|\leq t-R} |\nabla v(t,x)|^2+(\partial_tv(t,x))^2\,dx=0.  
 \end{equation} 
\end{claim}
This is classical. We postpone the proof after the proof of Lemma \ref{L:4} (see also Lemma 2.1 in \cite{DuJiKeMe16P}). 
\begin{proof}[Proof of Lemma \ref{L:4}]
 We let $F$ be as in Lemma \ref{L:3} and $g(\eta,\omega)=F(\eta,\omega,0)$. Note that $g\in C^{\infty}(\RR\times S^{N-1})$.

 \EMPH{Step 1} We prove that $g\in \dot{H}^1_{\eta}$, i.e.
 \begin{equation}
  \label{A9}
  \int_{\RR\times S^{N-1}} |\partial_{\eta}g|^2<\infty.
 \end{equation} 
 We have $F\left(r-t,\omega,\frac 1r\right)=r^{\frac{N-1}{2}}v(t,r\omega)$. As a consequence, $F(\eta,\omega,\sigma)=0$ if $\eta\geq M$ by finite speed of propagation, where $M>0$ is such that 
 \begin{equation}
  \label{A10}
  \supp (v_0,v_1)\subset\left\{x\in \RR^N\;:\; |x|\leq M\right\}.
 \end{equation} 
 Let $A>0$. We have 
 $$\int_{S^{N-1}}\int_{t-A}^{t+M} \left|\partial_{\eta}F\left( r-t,\omega,\frac{1}{r} \right) \right|^2\,drd\omega=\int_{S^{N-1}}\int_{t-A}^{t+M} r^{N-1}\left( \partial_tv(t,r\omega) \right)^2\,drd\omega\leq 2 E_{\lin}(v_0,v_1).$$
 Hence 
 $$\int_{S^{N-1}} \int_{-A}^M\left|\partial_{\eta}F\left( \eta,\omega,\frac{1}{\eta+t} \right)\right|^2\,d\eta d\omega\leq 2 E_{\lin}(v_0,v_1).$$
Letting $t\to+\infty$, we obtain
$$\int_{S^{N-1}} \int_{-A}^{+\infty} |\partial_\eta g(\eta,\omega)|^2\,d\eta d\omega=\int_{S^{N-1}} \int_{-A}^M \left|\partial_\eta F(\eta,\omega,0)\right|^2\,d\eta d\omega\leq 2 E_{\lin}(v_0,v_1).$$
Since $A$ can be taken arbitrarily large, we obtain \eqref{A9}.

\EMPH{Step 2} Proof of \eqref{A7}. By Claim \ref{Cl:5} and Step 1, it is sufficient to prove:
\begin{equation}
 \label{A:12}
 \forall R>0,\quad \lim_{t\to\infty} \int_{|x|\geq t-R} \frac{1}{|x|^2}|\nabla_{\omega}v|^2\,dx=0,
\end{equation} 
i.e. 
$$\forall R>0,\quad \lim_{t\to\infty} \int_{S^{N-1}}\int_{t-R}^{t+M} \frac{1}{r^2}\left|\nabla_{\omega}F\left( r-t,\omega,\frac{1}{r} \right)\right|^2\,dr d\omega=0,$$
which follows easily from the fact that $F\in C^{\infty}\left( \RR\times S^{N-1}\times [0,+\infty)\right)$ and the change of variable $\eta=r-t$.

\EMPH{Step 3} Proof of \eqref{A8}. By Claim \ref{Cl:5} and Step 1, it is sufficient to prove:
\begin{multline}
 \label{A:13}
 \forall R>0,\quad \lim_{t\to+\infty}\int_{S^{N-1}} \int_{t-R}^{t+M} \left| \frac{\partial}{\partial r}\left( F\left(r-t,\omega,\frac{1}{r} \right)-F(r-t,\omega,0 )\right)\right|^2\,dr d\omega\\+\int_{S^{N-1}} \int_{t-R}^{t+M} \left| \frac{\partial}{\partial t}\left( F\left(r-t,\omega,\frac{1}{r} \right)-F(r-t,\omega,0 )\right)\right|^2\,dr d\omega=0.
\end{multline} 
As in Step 2, this follows from the change of variable $\eta=r-t$.
 \end{proof}
\begin{lemma}
 \label{L:6}
 Let $(v_0,v_1)$ and $g$ be as in Lemma \ref{L:4}. Then
 $$E_{\lin}(v_0,v_1)=\left\|\overline{g}\right\|^2_{\dot{H}^1_\eta}.$$
\end{lemma}
\begin{proof}
\begin{multline*}
 E_{\lin}(v,\partial_tv)=\frac{1}{2}\int_0^{\infty} \int_{S^{N-1}} (\partial_rv)^2\,r^{N-1}d\omega dr\\
 +\frac{1}{2}\int_0^{\infty} \int_{S^{N-1}} \frac{1}{r^2}|\nabla_{\omega}v|^2r^{N-1}\,d\omega dr+\frac{1}{2}\int_0^{\infty} \int_{S^{N-1}} (\partial_tv)^2\,r^{N-1}d\omega dr. 
\end{multline*}

 Notice that 
 \begin{equation}
 \label{A12'}\int_0^{\infty} \int_{S^{N-1}} (\partial_rv)^2r^{N-1}\,d\omega dr=\int_0^{+\infty} \int_{S^{N-1}} \left( \partial_r (r^{\frac{N-1}{2}} v) \right)^2\,dr +o(1)\text{ as }t\to \infty.  
 \end{equation} 
Indeed, $\partial_r \left( r^{\frac{N-1}{2}}v \right)=r^{\frac{N-1}{2}}\partial_r v+\frac{N-1}{2} r^{\frac{N-3}{2}} v$,
and \eqref{A12'} follows in view of \eqref{A7}. Using \eqref{A7} again, we see that 
$$E_{\lin}(v,\partial_t v)=\frac{1}{2}\int_0^{\infty} \int_{S^{N-1}} \left(\partial_r (r^{\frac{N-1}{2}}v)\right)^2~\,d\omega dr+\frac{1}{2}\int_0^{\infty} \int_{S^{N-1}} (\partial_tv)^2\,r^{N-1}d\omega dr+o(1)\text{ as }t\to\infty.$$
By \eqref{A8}, 
\begin{multline*}
 E_{\lin}(v,\partial_t v)=\int_0^{\infty} \int_{S^{N-1}} \left|\partial_{\eta}g(r-t,\omega)\right|^2\,dr d\omega+o(1)
=\int_{-t}^{\infty} \int_{S^{N-1}} \left|\partial_{\eta}g(\eta,\omega)\right|^2\,d\eta d\omega+o(1)\\
\underset{t\to\infty}{\longrightarrow}\int_{-\infty}^{\infty} \int_{S^{N-1}} \left|\partial_{\eta}g(\eta,\omega)\right|^2\,d\eta d\omega,
\end{multline*}
and we conclude using the conservation of the energy. 
 \end{proof}
 
It remains to prove Claim \ref{Cl:5}.
We use the identity
$$ \frac{d}{dt} \left( \int x\cdot \nabla v \partial_t v+\frac{N-1}{2}\int v\partial_tv \right)=-E_{\lin}(v_0,v_1).$$
(In all the proof of the claim, $\int$ denotes the integral on $\RR^N$).
Integrating between $0$ en $t>0$, we obtain
\begin{equation}
 \label{A18} 
 \int x\cdot \nabla v \partial_t v+\frac{N-1}{2}\int v\partial_tv=-tE_{\lin}(v_0,v_1)+\int x\cdot \nabla v_0v_1+\frac{N-1}{2} \int v_0v_1.
\end{equation} 

Since $N\geq 3$, $\dot{H}^{-1}(\RR^N)$ is a Hilbert space, included in the space of tempered distributions, and $C_0^{\infty}(\RR^N)\subset \dot{H}^{-1}(\RR^N)$. By conservation of the $L^2\times \dot{H}^{-1}$ norm of $(v,\partial_tv)$, we deduce that the $L^2$ norm of $v$ remains bounded. Using \eqref{A18}, we deduce
\begin{equation}
 \label{A19} 
 \left| \int x\cdot \nabla v\partial_tv \right| \geq t E_{\lin}(v_0,v_1)-C.
\end{equation} 
For some constant $C>0$ depending on $(v_0,v_1)$.

Let $M>0$ such that $|x|\leq M$ on the support of $(v_0,v_1)$. By finite speed of propagation, $|x|\leq M+|t|$ on the support of $(v(t),\partial_tv(t))$. Let $R>0$. By Cauchy-Schwarz inequality,
\begin{multline*}
 \left| \int x\cdot \nabla v\partial_tv\right| \leq (t+M) \int_{|x|\geq t-R} |\nabla v\partial_tv|+(t-R)\int_{|x|\leq t-R} |\nabla v\partial_tv|\\
 \leq tE_{\lin}(v_0,v_1)+M\,E_{\lin}(v_0,v_1)-\frac R2 \int_{|x|\leq t-R} |\nabla_{t,x} v|^2.
\end{multline*}
Combining with \eqref{A19}, we obtain
\begin{equation}
 \label{A20}
 C+ME_{\lin}(v_0,v_1)\geq \frac{R}{2}\int_{|x|\leq t-R} |\nabla_{t,x} v|^2,
\end{equation} 
which yields as announced
\begin{equation}
 \label{A21} 
 \lim_{R\to \infty} \limsup_{t\to +\infty} \int_{|x|\leq t-R} |\nabla_{t,x}v|^2=0.
\end{equation} 
It remains to prove:
\begin{equation}
 \label{A22}
\lim_{t\to+\infty} \int \frac{1}{|x|^2}|v|^2=0.
\end{equation} 
Since 
$$ \sup_{t\geq 0} \int|v(t)|^2<\infty,$$
it is sufficient to prove
$$ \forall A>0,\quad \lim_{t\to +\infty} \int_{|x|\leq A} \frac{1}{|x|^2}|v|^2=0.$$
Let $\varphi\in C_0^{\infty}(\RR^N)$ such that $\varphi(x)=1$ if $|x|\leq 1$ and $\varphi(x)=0$ if $|x|\geq 2$. Then
$$ \int \frac{1}{|x|^2} \varphi\left( \frac{x}{\sqrt{t}} \right) |v|^2 \leq \int \left| \nabla\left( \varphi\left( \frac{x}{\sqrt{t}} \right)u \right)\right|^2\leq \frac{C}{t}\int |v|^2+2\int \left| \varphi\left( \frac{x}{\sqrt{t}} \right)\nabla v\right|^2.$$
The first term goes to $0$ as $t$ goes to infinity since the $L^2$ norm of $v$ is bounded. The second one goes to $0$ by \eqref{A21}. The proof is complete. \qed

\subsection{General case}
We prove here Theorem \ref{T:B2'}.
\subsubsection{Existence and uniqueness}
Let $(v_0,v_1)\in \dot{H}^1\times L^2$. We argue by density, considering a sequence $\left\{(v_{0,n},v_{1,n})\right\}_n$ in $\left( C_0^{\infty}(\RR^N) \right)^2$ such that 
\begin{equation}
 \label{A11}
 \lim_{n\to\infty} \left\|(v_{0,n}-v_0,v_{1,n}-v_1)\right\|_{\dot{H}^1\times L^2}=0.
\end{equation} 
Let $\overline{g}_n\in \dot{\HHH}^1_{\eta}$ be given by Lemma \ref{L:4}, corresponding to $(v_{0,n},v_{1,n})$. By the energy identity in Lemma \ref{L:6}, the sequence $\left\{\overline{g}_n\right\}_n$ is a Cauchy sequence in $\dot{\HHH}^1_{\eta}$. Since, by Proposition \ref{P:1},
$\dot{\HHH}^1_{\eta}$ is complete, we obtain that it has a limit $\overline{g}$ in $\dot{\HHH}^1_{\eta}$.

Let $\eps>0$. Choose $n$ such that 
\begin{equation}
 \label{A12}
 \left\|v_{0,n}-v_{0}\right\|_{\dot{H}^1}^2+\left\|v_{1,n}-v_1\right\|_{L^2}^2<\eps.
\end{equation} 
By conservation of the energy and \eqref{A7} in Lemma \ref{L:4}, we deduce
\begin{equation}
 \label{A13}
 \lim_{t\to\infty} \left\| \frac{1}{r}v(t)\right\|_{L^2}^2+\left\|\frac{1}{r}\nabla_{\omega} v(t)\right\|^2_{L^2}<C\eps,
\end{equation} 
which yields \eqref{A4}, letting $\eps\to 0$.

Fixing again $\eps$ and $n$ such that \eqref{A12} holds, and 
\begin{equation}
 \label{A13'}
 \left\|\overline{g}_n-\overline{g}\right\|^2_{\dot{\HHH}^1_{\eta}}<\eps,
\end{equation} 
we obtain, by \eqref{A12}, conservation of the energy and \eqref{A4},
\begin{equation}
 \label{A14}
 \limsup_{t\to\infty}
 \int_{[0,\infty)\times S^{N-1}}\left|\frac{\partial}{\partial r}\left( r^{\frac{N-1}{2}}v_n-r^{\frac{N-1}{2}}v \right)\right|^2+
\left|\frac{\partial}{\partial t}\left( r^{\frac{N-1}{2}}v_n-r^{\frac{N-1}{2}}v \right)\right|^2\,dr d\omega<\eps.
 \end{equation} 
 Hence, by the definition of $g_n$,
 \begin{equation}
  \label{A15}
  \limsup_{t\to\infty}
  \int_{[0,\infty)\times S^{N-1}}\left|\partial_{r,t}\left( g_n(r-t,\omega)-r^{\frac{N-1}{2}}v(t,r\omega) \right)\right|^2\,dr d\omega<\eps.
 \end{equation} 
 By \eqref{A13'} we can replace $g_n$ by $g$ in \eqref{A15} (changing $\eps$ into $4\eps$ in the right-hand side). Letting $\eps\to 0$, we obtain \eqref{A1}.
 
 Using that 
 \begin{equation*}
 \lim_{n\to\infty} E_{\lin}(v_{0,n},v_{1,n})=E_{\lin}(v_0,v_1)
 \text{ and }
  \lim_{n\to\infty} \left\|\overline{g}_n\right\|_{\dot{\HHH}^1_{\eta}}=\left\|\overline{g}\right\|_{\dot{\HHH}^1_{\eta}}
 \end{equation*}
we see that Lemma \ref{L:6} implies the energy identity \eqref{A2}. Of course \eqref{A2} implies the uniqueness of $\overline{g}$ and that the map $(v_0,v_1)\mapsto \sqrt{2}\overline{g}$ is an isometry from $\dot{H}^1\times L^2$ to $\HHH^1_{\eta}$.
\subsubsection{Proof of the surjectivity}
We next let $\overline{g}\in \dot{\HHH}^1_{\eta}$ and construct $v$ satisfying \eqref{E3} and such that \eqref{A1} holds. Since $E_{\lin}(v_0,v_1)=\|\overline{g}\|^2_{\dot{\HHH}^1_{\eta}}$, we see that the image of the map $(v_0,v_1)\mapsto \overline{g}$ is closed in $\dot{\HHH}^1_{\eta}$. Using the density of $C^{\infty}_0(\RR\times S^{N-1})$ in $\dot{\HHH}^1_{\eta}$, we see that it is sufficient to prove the existence of $v$ for $g\in C^{\infty}_0(\RR\times S^{N-1})$.

We look for $v$ of the form
\begin{equation}
 \label{A16}
 v(t,x)=\frac{1}{r^{\frac{N-1}{2}}} g(t-r,\omega)+\epsilon(t,x),
\end{equation} 
for large $t$, with $\omega=x/|x|$, and 
\begin{equation}
 \label{A16'}
 \lim_{t\to\infty} \left|(\epsilon,\partial_t\epsilon)(t)\right\|_{\hdot\times L^2}=0.
\end{equation}
Since 
$$(\partial_t^2-\Delta)\left( \frac{1}{r^{\frac{N-1}{2}}} g(t-r,\omega) \right)=\frac{N^2-4N+3}{4r^{\frac{N+3}{2}}} g(t-r,\omega)-\frac{1}{r^{\frac{N+3}{2}}}\Delta_{\omega}g(t-r,\omega)=:H(t,x),$$
the equation \eqref{E3} is equivalent to 
\begin{equation*}
(\partial_t^2-\Delta)\epsilon=-H.
\end{equation*}
Let $L\gg 1$ such that $\supp g\subset [-L,L]\times S^{N-1}$. Then 
$H$ is in $L^1\left([L+1,+\infty),L^2(\RR^N)\right)$. Indeed, if $t\geq L+1$,
\begin{multline*}
 \left\|\frac{1}{r^{\frac{N+3}{2}}}g(t-r,\omega)\right\|^2_{L^2(\RR^N)}+\left\| \frac{1}{r^{\frac{N+3}{2}}}\Delta_{\omega}g(t-r,\omega)\right\|^2_{L^2(\RR^N)}\\
 \lesssim \int_{t-L}^{t+L} \frac{1}{r^{N+3}}r^{N-1}\,dr\lesssim \frac{1}{(t-L)^3}-\frac{1}{(t+L)^3}\approx \frac{1}{t^4} \text{ as }t\to\infty.
\end{multline*}
Letting 
$$\epsilon(t,x)=\int_t^{+\infty}\frac{\sin\left( (t-s)\sqrt{-\Delta} \right)}{\sqrt{-\Delta}}H(s,x)\,ds,$$
we obtain $\epsilon$ satisfying \eqref{A16'} and such that $v$, defined by \eqref{A16}, is a solution of \eqref{E3}. 

By \eqref{A16} and \eqref{A16'},
$$\left\|\partial_{r,t}\left( r^{\frac{N-1}{2}}v-g(t-r,\omega) \right)\right\|_{L^2}=\left\|\partial_{t,r}\left( r^{\frac{N-1}{2}}\eps \right)\right\|_{L^2}\underset{t\to\infty}{\longrightarrow}0,$$
which concludes the proof. We used that 
$$\int_0^{+\infty}\left(\partial_r\left( r^{\frac{N-1}{2}}\epsilon \right)\right)^2\,dr=\int_0^{+\infty}\left( \partial_r\epsilon \right)^2r^{N-1}\,dr$$
if $\epsilon\in \dot{H}^1(\RR^N)$.

\subsection{Radiation fields and channels of energy}
\label{A:functionspace}
We conclude this appendix with a remark on the relation between exterior energy estimates (see Proposition 2.8 of \cite{DuKeMe12}) and radiation fields. 
Assume that $N$ is odd. It follows from the explicit formula for the solution of the wave equation in term of spherical means that the radiation field $G$ of a solution $v$ of the linear wave equation with initial data $(v_0,v_1)\in \hdot\times L^2$ has the following symmetry properties:
\begin{itemize}
 \item if $v_1=0$ then $G(\eta,\omega)=(-1)^{\frac{N+1}{2}}G(-\eta,-\omega)$.
 \item if $v_0=0$ then $G(\eta,\omega)=(-1)^{\frac{N-1}{2}}G(-\eta,-\omega)$.
\end{itemize}
(see for example the asymptotic formulas in the proof of Lemma 2.9 in \cite{DuKeMe12}). In both cases, we obtain
$$\lim_{t\to+\infty} \int_{|x|\geq |t|} |\nabla_{t,x} v(t,x)|^2\,dx=\int_{\RR}\int_{S^{N-1}} |G(\eta,\omega)|^2\,d\omega\,d\eta=
\frac 12\int|\nabla v_0(x)|^2+|v_1(x)|^2\,dx,$$
which yields the exterior energy estimate for the linear wave equation in odd space dimension (see Proposition 2.8 in \cite{DuKeMe12}).
\section{Study of a function space}
\label{S:funct}
We prove here Proposition \ref{P:1}.
\subsection{Completeness}
Let $\big\{\overline{g}_n\big\}_n$ be a Cauchy sequence in $\dot{\HHH}^1_{\eta}$. Replacing $g_n(\eta,\omega)$ by $g_n(\eta,\omega)-g_n(0,\omega)$ we can always assume $$\forall \omega\in S^{N-1},\quad g_n(0,\omega)=0.$$
As a consequence, if $\eta\in \RR$ and $n,p\in \NN$,
\begin{multline*}
\int_{S^{N-1}}\left|g_n(\eta,\omega)-g_p(\eta,\omega)\right|^2\,d\omega 
=\int_{S^{N-1}}\left|\int_0^{\eta} \left(\partial_{\eta}g_n(\eta',\omega)-\partial_{\eta}g_p(\eta',\omega)\right)d\eta'\right|^2\,d\omega \\
\leq |\eta|\int_{S^{N-1}}\int_0^{\eta} \left(\partial_{\eta}g_n(\eta',\omega)-\partial_{\eta}g_p(\eta',\omega)\right)^2d\eta'\,d\omega.
\end{multline*}
Thus, for all $M>0$, the sequence $\big\{g_n\big\}_n$ is a Cauchy sequence in $C^0\left( [-M,+M],L^2(S^{N-1}) \right)$. As a consequence, there exists $g\in C^0\left( \RR,L^2(S^{N-1}) \right)$ such that $\big\{g_n\big\}_n$ converges to $g$ locally in $L^{\infty}\left(\RR,L^2(S^{N-1}) \right)$.

Since $\big\{\partial_{\eta}g_n\big\}_n$ is a Cauchy sequence in $L^2\left( \RR\times S^{N-1} \right)$, it converges to some $h\in L^2\left( \RR\times S^{N-1} \right)$ in this space. Letting $n\to\infty$ in the equality
$$g_n(\eta,\omega)=\int_0^{\eta} \partial_{\eta}g_n(\eta',\omega)\,d\eta'\quad \text{for a.a. }\omega\in S^{N-1}$$
at fixed $\eta$, and taking the limit in $L^2(S^{N-1})$, we see that 
$$g(\eta,\omega)=\int_0^{\eta}h(\eta',\omega)\,d\eta'  \quad \text{for a.a. }\omega\in S^{N-1}.$$
Thus $\partial_{\eta}g=h\in L^2\left(\RR\times S^{N-1}\right)$ and, by the definition of $h$,
$$\lim_{n\to\infty}\left\| \partial_{\eta}g-\partial_{\eta}g_n\right\|_{L^2\left( \RR\times S^{N-1} \right)}=0.$$
\subsection{Density of compactly supported, smooth functions}
\EMPH{Step 1} We let $\varphi\in C^{\infty}_0(\RR)$ such that $\varphi(\eta)=1$ if $|\eta|\leq 1$ and $\varphi(\eta)=0$ if $|\eta|\geq 2$. Let $g\in \dot{H}^{1}_{\eta}$. In this step we prove:
$$ \lim_{R\to\infty} \left\|\frac{\partial}{\partial\eta}\left( g-\varphi\left( \frac{\cdot}{R} \right)g \right)\right\|_{L^2(\RR\times S^{N-1})}=0.$$
Indeed,
\begin{multline}
 \label{A46}
 \left\|\frac{\partial}{\partial\eta}\left( g-\varphi\left( \frac{\cdot}{R} \right)g \right)\right\|_{L^2(\RR\times S^{N-1})}\\
 \leq \left\|\left( 1-\varphi\left( \frac{\cdot}{R} \right)\right)\partial_{\eta}g \right\|_{L^2(\RR\times S^{N-1})}+\frac{1}{R}\left\|\varphi'\left( \frac{\cdot}{R}\right)g\right\|_{L^2(\RR\times S^{N-1})}.
\end{multline} 
The first term of the right-hand side goes to $0$ as $R$ goes to infinity by dominated convergence. We next treat the second term. We have:
\begin{equation}
 \label{A47}
 \frac{1}{R^2}\left\|\varphi'\left( \frac{\cdot}{R} \right)g\right\|_{L^2}^2\leq \frac{C}{R^2}\int_{R}^{2R} \int_{S^{N-1}} |g(\eta,\omega)|^2,d\omega d\eta.
\end{equation} 
Let $A\gg 1$. Then, for $\eta>A$,
$$ g(\eta,\omega)=g(A,\omega)+\int_A^{\eta} \partial_{\eta}g(\eta',\omega)\,d\eta'\quad \text{for a.a. }\omega\in S^{N-1}.$$
Hence, 
$$|g(\eta,\omega)|^2\leq 2\eta \int_{A}^{+\infty} |\partial_{\eta} g(\eta',\omega)|^2\,d\eta'+2|g(A,\omega)|^2\quad \text{for a.a. }\omega\in S^{N-1}.$$
Integrating on $S^{N-1}$, we obtain
$$\int_{S^{N-1}} |g(\eta,\omega)|^2\,d\omega\leq 2\eta \left\|\partial_{\eta} g\right\|^2_{L^2\left([A,+\infty)\times S^{N-1}\right)}+2\left\|g(A,\cdot)\right\|^2_{L^2(S^{N-1})}.$$
For $R>A$, we have 
$$\frac{1}{R^2}\int_{R}^{2R}\int_{S^{N-1}} |g(\eta,\omega)|^2\,d\omega d\eta\leq 4\left\|\partial_{\eta}g\right\|^2_{L^2([A,+\infty)\times S^{N-1})}+\frac{2}{R}\left\|g(A,\cdot)\right\|^2_{L^2(S^{N-1})}.$$
Letting $R\to \infty$, we obtain
$$\limsup_{R\to\infty} \frac{1}{R^2}\int_{R}^{2R} \int_{S^{N-1}}|g(\eta,\omega)|^2\,d\omega d\eta\leq 4\left\|\partial_{\eta}g\right\|^2_{L^2([A,+\infty)\times S^{N-1})}$$
and thus, since $A$ is arbitrarily large,
$$\lim_{R\to\infty} \frac{1}{R^2}\int_{R}^{2R} \int_{S^{N-1}}|g(\eta,\omega)|^2\,d\omega d\eta=0.$$
In view of \eqref{A46}, \eqref{A47}, this concludes Step 1.

\EMPH{Step 2} It remains to prove that a compactly supported function $g\in C^{0}\left(\RR,L^2(S^{N-1})\right)$ such that $\partial_{\eta} g\in L^2(\RR\times S^{N-1})$ can be approximated in the $\dot{\HHH}^1_{\eta}$ norm by smooth functions. This can be done using convolution with approximations of the identity (in the variable $\eta$) and projecting on the $n$ first eigenspaces of the Laplace-Beltrami operator $\Delta_{\omega}$ on $S^{N-1}$. We leave the details to the reader. 
\subsection{Isometry with $L^2$}
Let $\Phi: \overline{g}\mapsto \partial_{\eta}g$. It follows obviously from the definition of $\dot{H}^1_{\eta}$ that $\Phi$ is a well-defined, injective isometry. If $G\in L^2(\RR\times S^{N-1})$, then
$$ g(\eta,\omega)=\int_0^{\eta} G(\omega,\eta')\,d\omega$$
is an element of $\dot{H}^1_{\eta}$ which satisfies $\Phi(\overline{g})=G$, which proves that the map is also surjective. The proof of Proposition \ref{P:1} is complete.

\section{Profiles and estimates on Strichartz norms}
Recall that $S(\RR)=L^{\frac{N+2}{N-2}}\left(\RR,L^{\frac{2(N+2)}{N-2}}(\RR^N)\right)$.
\begin{claim}
 \label{Cl:A1}
 Let $\left\{(u_{0,n},u_{1,n})\right\}_n$ be a sequence in $\hdot\times L^2$ such that
 \begin{equation}
  \label{EA1}
  \limsup_{n\to\infty} \left\|(u_{0,n},u_{1,n})\right\|_{\hdot\times L^2}=M.
 \end{equation} 
 Assume furthermore
 \begin{equation}
  \label{EA2}
  \forall j,\quad \left\| U^j_{\lin}\right\|_{S(\RR)}\leq \delta_0.
 \end{equation} 
 Then 
 $$ \limsup_{n\to\infty} \left\| S_{\lin}(\cdot)(u_{0,n},u_{1,n})\right\|_{S(\RR)} \leq 
 \begin{cases}
 C M^{\frac{3}{4}}\delta_0^{\frac{1}{4}}& \text{ if } N=3\\
 C M^{\frac{2N-4}{N+2}}\delta_0^{\frac{6-N}{N+2}}&\text{ if } N=4,5.
 \end{cases}
 $$
\end{claim}
\begin{proof}
 We use the following Pythagorean expansion:
 \begin{equation}
  \label{EA3}
  \lim_{n\to\infty} \left\|S_{\lin}(\cdot)(u_{0,n},u_{1,n})\right\|^{\frac{2(N+1)}{N-2}}_{L^{\frac{2(N+1)}{N-2}}(\RR\times\RR^N)}=\sum_{j\geq 1}\left\|U^j_{\lin}\right\|^{\frac{2(N+1)}{N-2}}_{L^{\frac{2(N+1)}{N-2}}(\RR\times\RR^N)}.
 \end{equation} 
 Next, we notice that by H\"older and Sobolev inequality,
 \begin{equation}
  \label{EA4}
  \left\|U^j_{\lin}\right\|^{\frac{2(N+1)}{N-2}}_{L^{\frac{2(N+1)}{N-2}}(\RR\times\RR^N)}\leq \left\|U_{\lin}^j\right\|_{L^{\infty}\left(\RR,L^{\frac{2N}{N-2}}\right)}^{\frac{N}{N-2}}\left\|U_{\lin}^j\right\|_{S(\RR)}^{\frac{N+2}{N-2}}.
 \end{equation} 
 If $N=3$, we deduce
 \begin{equation*}
  \left\|U_{\lin}^j\right\|^8_{L^8_{t,x}}\leq C\left\|\vU_{\lin}^j(0)\right\|^2_{\hdot\times L^2}M\delta_0^5.
 \end{equation*} 
 Hence, by the Pythagorean expansion \eqref{EA3},
 \begin{equation*}
  \limsup_{n\to \infty}
  \left\|S_{\lin}(\cdot)(u_{0,n},u_{1,n})\right\|^8_{L^8_{t,x}}\leq CM\delta_0^5 \sum_{j\geq 1}\left\|\vU_{\lin}^j(0)\right\|^2_{\hdot\times L^2}\leq C\,M^3\delta_0^5.
 \end{equation*} 
 By H\"older's inequality, then Strichartz inequality
 \begin{align*}
  \left\|S_{\lin}(\cdot)(u_{0,n},u_{1,n})\right\|_{L^5_tL^{10}_x}&\leq \left\|S_{\lin}(\cdot)(u_{0,n},u_{1,n})\right\|_{L^4_tL^{12}_x}^{3/5}\left\|S_{\lin}(\cdot)(u_{0,n},u_{1.n})\right\|_{L^8_{t,x}}^{\frac{2}{5}}\\
  \limsup_{n\to\infty}\left\|S_{\lin}(\cdot)(u_{0,n},u_{1,n})\right\|_{L^5_tL^{10}_x}&\leq C M^{\frac 35} M^{\frac{3}{8}\times \frac{2}{5}}\delta_0^{\frac{5}{8}\times \frac{2}{5}}=C M^{\frac{3}{4}}\delta_0^{\frac 14}.
 \end{align*}
If $N\in \{4,5\}$, we combine \eqref{EA4} with Strichartz and Sobolev inequalities and obtain
\begin{equation}
 \label{EA5}
 \left\|U^j_{\lin}\right\|_{L^{\frac{2(N+1)}{N-2}}_{t,x}}^{\frac{2(N+1)}{N-2}}\leq C\left\|\vU^j_{\lin}(0)\right\|_{\hdot\times L^2}^2\left\|U^j_{\lin}\right\|_{S(\RR)}^{\frac{6}{N-2}}\leq C\left\|\vU^j_{\lin}(0)\right\|^2_{\hdot\times L^2} \delta_0^{\frac{6}{N-2}}.
\end{equation} 
Summing up, we deduce
\begin{equation}
 \label{EA6}
 \sum_{j\geq 1} \left\|U_{\lin}^j\right\|_{L^{\frac{2(N+1)}{N-2}}_{t,x}}^{\frac{2(N+1)}{N-2}}\leq C M^2\delta_0^{\frac{6}{N-2}}.
\end{equation} 
Hence
\begin{equation}
 \label{EA7}
 \limsup_{n\to \infty} \left\| S_{\lin}(\cdot)(u_{0,n},u_{1,n})\right\|_{L^{\frac{2(N+1)}{N-2}}_{t,x}}^{\frac{2(N+1)}{N-2}} \leq C M^2\delta_0^{\frac{6}{N-2}}.
\end{equation} 
By H\"older's inequality,
\begin{equation}
 \left\| S_{\lin}(\cdot)(u_{0,n},u_{1,n})\right\|_{S(\RR)}\leq  \left\|S_{\lin}(\cdot)(u_{0,n},u_{1,n})\right\|_{L^{\frac{2(N+1)}{N-2}}_{t,x}}^{\theta}\left\| S_{\lin}(\cdot)(u_{0,n},u_{1,n})\right\|_{L^2\left(\RR,L^{\frac{2N}{N-3}}\right)}^{1-\theta},
\end{equation} 
where $\theta=\frac{(N+1)(6-N)}{3(N+2)}$, $1-\theta=\frac{N(N-2)}{3(N+2)}$.

Hence
$$\left\| S_{\lin}(\cdot)(u_{0,n},u_{1,n})\right\|_{S(\RR)}\leq C\delta_0^{\frac{6-N}{N+2}}M^{\frac{2N-4}{N+2}}.$$
\end{proof}
\section{Dispersion for solution with small $e_1$-norm}
In this appendix we prove Claim \ref{Cl:E15}. We argue by contradiction. Assume that for all $n>0$, there exists $(v_{0,n},v_{1,n})\in \hdot\times L^2$ such that 
\begin{equation}
 \label{EA8}
 \left\|(v_0^n,v_1^n)\right\|_{e_1}\leq \frac 1n, \quad \left\|(v_{0,n},v_{1,n})\right\|_{\hdot\times L^2}\leq M
\end{equation} 
and 
\begin{equation}
\label{EA9}
\left\|v_{\lin}^n\right\|_{S(\RR)}>\beta,     
\end{equation} 
where $v_{\lin}^n(t)=S_{\lin}(t)\left( v_0^n,v_1^n \right)$. Extracting subsequences, we can assume that $\left\{(v_0^n,v_1^n)\right\}$ has a profile decomposition \profiles. By the Pythagorean expansion of the norm $\|\cdot\|_{e_1}$ (see Remark \ref{R:E14'''}), for all $J\geq 1$,
\begin{equation*}
 \frac{1}{n}\geq\left\|(v_0^n,v_1^n)\right\|_{e_1}^2=\sum_{j=1}^J \left\|\vU_{\lin}^j(0)\right\|^2_{e_1}+\left\|w_n^J(0)\right\|^2_{e_1}+o(1)
\end{equation*} 
as $n\to \infty$. As a consequence, $U^j_{\lin}=0$ for all $j\geq 1$ and $\lim_n \|v_{\lin}^n\|_{S(\RR)}=0$, which contradicts \eqref{EA9}.
\bibliographystyle{acm}
\bibliography{toto}

\begin{thebibliography}{10}

\bibitem{Alinhac09Bo}
{\sc Alinhac, S.}
\newblock {\em Hyperbolic partial differential equations}.
\newblock Universitext. Springer, Dordrecht, 2009.

\bibitem{BaGe99}
{\sc Bahouri, H., and G{\'e}rard, P.}
\newblock High frequency approximation of solutions to critical nonlinear wave
  equations.
\newblock {\em Amer. J. Math. 121}, 1 (1999), 131--175.

\bibitem{BaSh98}
{\sc Bahouri, H., and Shatah, J.}
\newblock Decay estimates for the critical semilinear wave equation.
\newblock {\em Ann. Inst. H. Poincar\'e Anal. Non Lin\'eaire 15}, 6 (1998),
  783--789.

\bibitem{Bulut10}
{\sc Bulut, A.}
\newblock Maximizers for the {S}trichartz inequalities for the wave equation.
\newblock {\em Differential Integral Equations 23}, 11/12 (2010), 1035--1072.

\bibitem{CoKeLaSc15P}
{\sc C\^ote, R., Kenig, C.~E., Lawrie, A., and Schlag, W.}
\newblock Profiles for the radial focusing 4d energy-critical wave equation.
\newblock ArXiv preprint:1402.2307.

\bibitem{CoKeLaSc15b}
{\sc C{\^o}te, R., Kenig, C.~E., Lawrie, A., and Schlag, W.}
\newblock Characterization of large energy solutions of the equivariant wave
  map problem: {II}.
\newblock {\em Amer. J. Math. 137}, 1 (2015), 209--250.

\bibitem{Ding86}
{\sc Ding, W.~Y.}
\newblock On a conformally invariant elliptic equation on {${\bf R}^n$}.
\newblock {\em Comm. Math. Phys. 107}, 2 (1986), 331--335.

\bibitem{DuJiKeMe16P}
{\sc Duyckaerts, T., Jia, H., Kenig, C.~E., and Merle, F.}
\newblock Soliton resolution along a sequence of times for the focusing
  energy-critical wave equation, 2016.

\bibitem{DuKeMe12b}
{\sc Duyckaerts, T., Kenig, C., and Merle, F.}
\newblock Profiles of bounded radial solutions of the focusing, energy-critical
  wave equation.
\newblock {\em Geom. Funct. Anal. 22}, 3 (2012), 639--698.

\bibitem{DuKeMe12}
{\sc Duyckaerts, T., Kenig, C., and Merle, F.}
\newblock Universality of the blow-up profile for small type {II} blow-up
  solutions of the energy-critical wave equation: the nonradial case.
\newblock {\em J. Eur. Math. Soc. (JEMS) 14}, 5 (2012), 1389--1454.

\bibitem{DuKeMe13}
{\sc Duyckaerts, T., Kenig, C., and Merle, F.}
\newblock Classification of radial solutions of the focusing, energy-critical
  wave equation.
\newblock {\em Cambridge Journal of Mathematics 1}, 1 (2013), 75--144.

\bibitem{DuKeMe15Pb}
{\sc Duyckaerts, T., Kenig, C., and Merle, F.}
\newblock Concentration-compactness and universal profiles for the non-radial
  energy critical wave equation.
\newblock ArXiv preprint:1510.01750, 2015.

\bibitem{Friedlander62}
{\sc Friedlander, F.~G.}
\newblock On the radiation field of pulse solutions of the wave equation.
\newblock {\em Proc. Roy. Soc. Ser. A 269\/} (1962), 53--65.

\bibitem{Friedlander80}
{\sc Friedlander, F.~G.}
\newblock Radiation fields and hyperbolic scattering theory.
\newblock {\em Math. Proc. Cambridge Philos. Soc. 88}, 3 (1980), 483--515.

\bibitem{GiSoVe92}
{\sc Ginibre, J., Soffer, A., and Velo, G.}
\newblock The global {C}auchy problem for the critical nonlinear wave equation.
\newblock {\em J. Funct. Anal. 110}, 1 (1992), 96--130.

\bibitem{GiVe95}
{\sc Ginibre, J., and Velo, G.}
\newblock Generalized {S}trichartz inequalities for the wave equation.
\newblock {\em J. Funct. Anal. 133}, 1 (1995), 50--68.

\bibitem{Grillakis90b}
{\sc Grillakis, M.~G.}
\newblock Regularity and asymptotic behaviour of the wave equation with a
  critical nonlinearity.
\newblock {\em Ann. of Math. (2) 132}, 3 (1990), 485--509.

\bibitem{Grillakis92}
{\sc Grillakis, M.~G.}
\newblock Regularity for the wave equation with a critical nonlinearity.
\newblock {\em Comm. Pure Appl. Math. 45}, 6 (1992), 749--774.

\bibitem{JiaLiuXu14P}
{\sc Jia, H., Liu, B., and Xu, G.}
\newblock Long time dynamics of defocusing energy critical 3+ 1 dimensional
  wave equation with potential in the radial case.
\newblock {\em arXiv preprint arXiv:1403.5696\/} (2014).

\bibitem{Kapitanski94}
{\sc Kapitanski, L.}
\newblock Global and unique weak solutions of nonlinear wave equations.
\newblock {\em Math. Res. Lett. 1}, 2 (1994), 211--223.

\bibitem{KeMe08}
{\sc Kenig, C.~E., and Merle, F.}
\newblock Global well-posedness, scattering and blow-up for the energy-critical
  focusing non-linear wave equation.
\newblock {\em Acta Math. 201}, 2 (2008), 147--212.

\bibitem{LaPh89}
{\sc Lax, P.~D., and Phillips, R.~S.}
\newblock {\em Scattering theory}, second~ed., vol.~26 of {\em Pure and Applied
  Mathematics}.
\newblock Academic Press Inc., Boston, MA, 1989.
\newblock With appendices by Cathleen S. Morawetz and Georg Schmidt.

\bibitem{Levine74}
{\sc Levine, H.~A.}
\newblock Instability and nonexistence of global solutions to nonlinear wave
  equations of the form {$Pu_{tt}=-Au+\mathcal{F}(u)$}.
\newblock {\em Trans. Amer. Math. Soc. 192\/} (1974), 1--21.

\bibitem{Nakanishi99b}
{\sc Nakanishi, K.}
\newblock Scattering theory for the nonlinear {K}lein-{G}ordon equation with
  {S}obolev critical power.
\newblock {\em Internat. Math. Res. Notices}, 1 (1999), 31--60.

\bibitem{ShSt93}
{\sc Shatah, J., and Struwe, M.}
\newblock Regularity results for nonlinear wave equations.
\newblock {\em Ann. of Math. (2) 138}, 3 (1993), 503--518.

\bibitem{ShSt94}
{\sc Shatah, J., and Struwe, M.}
\newblock Well-posedness in the energy space for semilinear wave equations with
  critical growth.
\newblock {\em Internat. Math. Res. Notices}, 7 (1994), 303ff., approx.\ 7 pp.\
  (electronic).

\bibitem{Struwe88}
{\sc Struwe, M.}
\newblock Globally regular solutions to the {$u^5$} {K}lein-{G}ordon equation.
\newblock {\em Ann. Scuola Norm. Sup. Pisa Cl. Sci. (4) 15}, 3 (1988), 495--513
  (1989).

\bibitem{Tao04DPDE}
{\sc Tao, T.}
\newblock On the asymptotic behavior of large radial data for a focusing
  non-linear {S}chr\"odinger equation.
\newblock {\em Dyn. Partial Differ. Equ. 1}, 1 (2004), 1--48.

\bibitem{Tao07DPDE}
{\sc Tao, T.}
\newblock A (concentration-) compact attractor for high-dimensional non-linear
  schr{\"o}dinger equations.
\newblock {\em Dyn. Partial Differ. Equ. 4}, 1 (2007), 1--53.

\bibitem{Tao08DPDE}
{\sc Tao, T.}
\newblock A global compact attractor for high-dimensional defocusing non-linear
  {S}chr\"odinger equations with potential.
\newblock {\em Dyn. Partial Differ. Equ. 5}, 2 (2008), 101--116.

\end{thebibliography}
\end{document}